\documentclass[a4paper,10pt]{article}
\usepackage{amsmath,amssymb,amsfonts,amsthm}
\usepackage{fouridx}
\usepackage{a4wide}
\usepackage[usenames,dvipsnames]{color}
\usepackage[verbose,colorlinks=true,linktocpage=true,linkcolor=blue,citecolor=blue]{hyperref}

\theoremstyle{definition}
\newtheorem{definition}{Definition}[section]

\theoremstyle{plain}
\newtheorem{theorem}[definition]{Theorem}
\newtheorem{proposition}[definition]{Proposition}
\newtheorem{lemma}[definition]{Lemma}
\newtheorem{corollary}[definition]{Corollary}
\newtheorem{example}[definition]{Example}
\theoremstyle{remark}
\newtheorem{remark}[definition]{Remark}

\numberwithin{equation}{section}

\allowdisplaybreaks[4]

\begin{document}

\title{Second Homology of  Generalized Periplectic Lie Superalgebras}
\author{Zhihua Chang ${}^1$, Jin Cheng ${}^2$ and Yongjie Wang ${}^3$\footnote{Corresponding Author: Yongjie Wang, Email: yjwang@ncts.ntu.edu.tw}}
\maketitle

\begin{center}
\footnotesize
\begin{itemize}
\item[1] School of Mathematics, South China University of Technology, Guangzhou, Guangdong, 510640, China.
\item[2] School of Mathematics and Statistics, Shandong Normal University, Jinan, Shandong, 250014, China.
\item[3] Mathematics Division, National Center for Theoretical Sciences, Taipei, 10617, Taiwan.
\end{itemize}
\end{center}

\begin{abstract}
Let $(R,{}^-)$ be an arbitrary unital associative superalgebra with superinvolution over a commutative ring $\Bbbk$ with $2$ invertible. The second homology of the generalized periplectic Lie superalgebra $\mathfrak{p}_m(R,{}^-)$ for $m\geqslant3$ has been completely determined via an explicit construction of its universal central extension. In particular, this second homology is identified with the first $\mathbb{Z}/2\mathbb{Z}$-graded dihedral homology of $R$ with certain superinvolution whenever $m\geqslant5$.
\bigskip

\bigskip

\end{abstract}

\section{Introduction}
\label{sec:intr}
It is well known that the second homology of a Lie (super)algebra $\mathfrak{g}$ is identified with the kernel of its universal central extension, and thus classifies all central extensions of $\mathfrak{g}$ up to isomorphism (c.f. \cite{Neher2003,ScheunertZhang1998}). They play crucial roles in the theory of Lie (super)algebras.

A remarkable work about the second homology of a Lie algebra is the nice connection between the second homology of a matrix Lie algebra and the first cyclic homology of its coordinates associative algebra established in \cite{KasselLoday1982}. Let $A$ be a unital associative algebra over a commutative ring $\Bbbk$ with $2$ invertible. One denotes $\mathfrak{gl}_n(A)$ the Lie algebra of all $n\times n$-matrices with entries in $A$ under commutator operation and $\mathfrak{sl}_n(A)$ the derived Lie subalgebra of $\mathfrak{gl}_n(A)$. It is shown in \cite{KasselLoday1982} that the second homology $\mathrm{H}_2(\mathfrak{sl}_n(A))$ with $n\geqslant2$ is isomorphic to the first cyclic homology $\mathrm{HC}_1(A)$. Such an isomorphism has been extended to many other classes of Lie (super)algebras. For instance, Y. Gao showed in \cite{Gao1996} that the second homology of elementary unitary Lie algebra $\mathfrak{eu}_n(R,{}^-)$ with $n\geqslant5$ is identified with the first skew-dihedral homology of $(R,{}^-)$ that is a unital associative algebra with anti-involution. 

In the particular case where the coordinates algebra is associative and commutative, the matrix Lie algebras under consideration are isomorphic to current Lie algebras of the form $\mathfrak{g}\otimes A$, where $\mathfrak{g}$ is a finite dimensional (simple) Lie algebra and $A$ is an associative commutative algebra. For instance, $\mathfrak{sl}_n(A)$ is isomorphic to $\mathfrak{sl}_n(\Bbbk)\otimes_{\Bbbk}A$ if $A$ is commutative, while $\mathfrak{eu}_n(R,\mathrm{id})$ is isomorphic to $\mathfrak{eu}_n(\Bbbk,\mathrm{id})\otimes_{\Bbbk}R$ if $R$ is commutative. Low-degree (co)homology of current Lie algebras were also studied intensively. In \cite{NW2008}, K.-H. Neeb and F. Wagemann explicitly described the second cohomology of current algebras of general Lie algebras with coefficients in a trivial module. More details could be found in \cite{NW2008} and the refenrences therein.

The super analogue of C. Kassel and J. L. Loday's work was obtained in \cite{ChenGuay2013,ChenSun2015}.  The isomorphism between the second homology of the Lie superalgebra $\mathfrak{sl}_{m|n}(S)$ coordinated by a unital associative superalgebra $S$ with $m+n\geqslant5$ and the first $\mathbb{Z}/2\mathbb{Z}$-graded cyclic homology $\mathrm{HC}_1(S)$ was established. Recent investigation \cite{ChangWang2016} further gave the identification between the second homology of the ortho-symplectic Lie superalgebra $\mathfrak{osp}_{m|2n}(R,{}^-)$ and the first $\mathbb{Z}/2\mathbb{Z}$-graded skew-dihedral homology of $(R,{}^-)$ for $(m,n)\neq(1,1)$ or $(2,1)$, where $(R,{}^-)$ is a unital associative superalgebra with superinvolution (see (\ref{eq:supinv}) for the definition). A series of deep investigations on the relationship between the homology theory of Lie algebras and the homology theory of associative algebras have been made in \cite{LodayProcesi1988, LodayQuillen1984}.

Inspired by the above developments, we aim to establish an isomorphism that is analogous to C. Kassel and J. L. Loday's isomorphism for the generalized periplectic Lie superalgebra $\mathfrak{p}_m(R,{}^-)$ coordinatized by a unital associative superalgebra $(R,{}^-)$ with superinvolution. As in Section~\ref{sec:palg}, a generalized periplectic Lie superalgebra is defined as the derived sub-superalgebra of the Lie superalgebra of all skew-symmetric matrices with respect to the so-called periplectic superinvolution. It is a super analogue of a unitary Lie algebra introduced in \cite{AllisonFaulkner1993}. This family of Lie superalgebras provides us with a realization of an arbitrary generalized root graded Lie superalgebra of type $P(m-1)$ for $m\neq4$ up to central isogeny (cf. \cite{ChengGao2015}), which is a complement to the realization of a root graded Lie superalgebra of type $P(m-1)$ given in \cite{MartinezZelmanov2003}.

A primary result of this paper is Theorem~\ref{thm:hml_pm} which states that the second homology of the Lie superalgebra $\mathfrak{p}_m(R,{}^-)$ with $m\geqslant5$ is isomorphic to the first $\mathbb{Z}/2\mathbb{Z}$-graded dihedral homology of $(R,{}^-\circ\rho)$, where ${}^-\circ\rho$ is the superinvolution on $R$ obtained by twisting the superinvolution ${}^-$ with the sign map $\rho$ (see (\ref{eq:rho}) in Section~\ref{sec:palg}). In the special case where $R$ is super-commutative, the isomorphism indicates that the second homology of $\mathfrak{p}_m(\Bbbk)\otimes_{\Bbbk}R$ for a super-commutative superalgebra $R$ is trivial, which was obtained by K. Iohara and Y. Koga in \cite{IoharaKoga2001, IoharaKoga2005}. While the isomorphism also reveals that the second homology  of $\mathfrak{p}_m(R,{}^-)$ is not necessarily trivial if $R$ is not super-commutative. 

The methods used in this paper unsurprisingly involve an explicit construction of the universal central extension of $\mathfrak{p}_m(R,{}^-)$, which will be achieved via introducing the notion of the Steinberg periplectic Lie superalgebra $\mathfrak{stp}_m(R,{}^-)$ in Section~\ref{sec:stp}.

The isomorphism between the second homology of $\mathfrak{p}_m(R,{}^-)$ and the first $\mathbb{Z}/2\mathbb{Z}$-graded dihedral homology of $(R,{}^-\circ\rho)$ fails when $m=3$ or $4$. Nonetheless, the second homology of $\mathfrak{p}_4(R,{}^-)$ and $\mathfrak{p}_3(R,{}^-)$ will also be explicitly characterized in Theorems~\ref{thm:p4_hml} and~\ref{thm:p3_hml}.

Throughout this paper, we always assume that $\Bbbk$ is a commutative base ring with $2$ invertible, and suppose that all modules, associative superalgebras and Lie superalgebras are defined over $\Bbbk$. We also use the following degree convention: if $v$ is an element of a $\mathbb{Z}/2\mathbb{Z}$-graded vector space and its degree $|v|$ appears in some formula or expression, then $v$ is assumed to be homogeneous.

 \section{Generalized Periplectic Lie Superalgebras}
 \label{sec:palg}
 
We briefly review the definition of a generalized periplectic Lie superalgebra and prove a few of its properties in this section.

Let $R$ be a unital associative superalgebra. Then the associative superalgebra $\mathrm{M}_{m|m}(R)$ of all $2m\times 2m$-matrices is also equipped with a $\mathbb{Z}/2\mathbb{Z}$-grading by setting 
\begin{equation*}
|e_{ij}(a)|:=|i|+|j|+|a|,\quad a\in R,\quad 1\leqslant i,j\leqslant 2m,
\label{eq:parity_matrix}
\end{equation*}
where $e_{ij}(a)$ is the matrix unit with $a$ at the $(i,j)$-position and $0$ elsewhere, and
\begin{equation*}
|i|=\begin{cases}0,&\text{if }i\leqslant m,\\1,&\text{if }i>m.\end{cases}
\label{eq:parity_index}
\end{equation*}
This makes $\mathrm{M}_{m|m}(R)$ an associative superalgebra. 

We assume in addition that $R$ is equipped with a superinvolution\footnote{The superinvolutions on the matrix superalgebra $\mathrm{M}_{m|n}(\Bbbk)$ over a field $\Bbbk$ of characteristic not $2$ are classified in \cite{Racine1998}. A superinvolution on $\mathrm{M}_{m|n}(\Bbbk)$ may not exist. Whenever it exists, a superinvolution on $\mathrm{M}_{m|n}(\Bbbk)$ is equivalent to either a periplectic superinvolution or an ortho-symplectic superinvolution. This motivates us to define the periplectic superinvolution on $\mathrm{M}_{m|m}(R)$ here.} ${}^-:R\rightarrow R$ that is a $\Bbbk$-linear map satisfying
\begin{equation}
\overline{ab}=(-1)^{|a||b|}\bar{b}\bar{a},\text{ and }\bar{\bar{a}}=a,
\label{eq:supinv}
\end{equation}
for $a,b\in R$.  This further gives rise to a periplectic superinvolution on the associative superalgebra $\mathrm{M}_{m|m}(R)$ defined by
\begin{equation*}
\begin{pmatrix}A&B\\C&D\end{pmatrix}^{\mathrm{prp}}:=
\begin{pmatrix}\overline{D}^t&-\overline{\rho(B)}^t\\\overline{\rho(C)}^t&\overline{A}^t\end{pmatrix},
\end{equation*}
where $A,B,C,D$ are $m\times m$-matrices with entries in $R$,  $\rho:R\rightarrow R$ is the $\Bbbk$-linear map defined by
\begin{equation}
\rho(a)=(-1)^{|a|}a,\label{eq:rho}
\end{equation}
for $a\in R$, $\rho(A)$ denotes the matrix $(\rho(a_{ij}))$ and $\overline{A}=(\overline{a_{ij}})$ for $A=(a_{ij})$. In this situation, one defines a Lie superalgebra 
\begin{equation*}
\widetilde{\mathfrak{p}}_m(R,{}^-):=\{X\in\mathrm{M}_{m|m}(R)|X^{\mathrm{prp}}=-X\},
\end{equation*}
with the standard super-commutator as the super-bracket, whose derived Lie sub-superalgebra
\begin{equation*}
\mathfrak{p}_m(R,{}^-):=[\widetilde{\mathfrak{p}}_m(R,{}^-),\widetilde{\mathfrak{p}}_m(R,{}^-)]
\end{equation*}
is called \textit{the generalized periplectic Lie superalgebra coordinatized by the associative superalgebra $(R,{}^-)$ with superinvolution}.
\medskip

As an example, we consider $R=\Bbbk$ on which the identity map is a superinvolution. The Lie superalgebra $\mathfrak{p}_m(\Bbbk,\mathrm{id})$ coincides with the simple Lie superalgebra of type $P(m-1)$ as defined in \cite{Kac1977}. We simply write $\mathfrak{p}_m(\Bbbk):=\mathfrak{p}_m(\Bbbk,\mathrm{id})$.

If $R$ is super-commutative, there is a Lie superalgebra structure on $\mathfrak{p}_m(\Bbbk)\otimes_{\Bbbk}R$ given by
\begin{equation*}
[x\otimes a,y\otimes b]=(-1)^{|a||y|}[x,y]\otimes ab,
\end{equation*}
for $x,y\in\mathfrak{p}_m(\Bbbk)$ and $a,b\in R$. This Lie superalgebra is isomorphic to the generalized periplectic Lie superalgebra $\mathfrak{p}_m(R,\rho)$, where the $\Bbbk$-linear map $\rho:R\rightarrow R$ (\ref{eq:rho}) is a superinvolution on a super-commutative superalgebra $R$.
\medskip

Before going into the discussion on the properties of generalized periplectic Lie superalgebras, we exhibit another example here.
\begin{example}
\label{ex:iso_p_SS}
Let $S$ be an arbitrary unital associative superalgebra and $S^{\mathrm{op}}$ be its \textit{opposite superalgebra} with the multiplication
\begin{equation*}
a\overset{\mathrm{op}}{\cdot}b=(-1)^{|a||b|}b\cdot a,
\end{equation*}
for $a,b\in S$. Then the $\Bbbk$-linear map
\begin{equation*}
\mathrm{ex}:S\oplus S^{\mathrm{op}}\rightarrow S\oplus S^{\mathrm{op}}, \quad a\oplus b\mapsto b\oplus a.
\end{equation*}
is a superinvolution on $S\oplus S^{\mathrm{op}}$. In this situation, we have an isomorphism of Lie superalgebras
\begin{equation*}
\mathfrak{p}_m(S\oplus S^{\mathrm{op}},\mathrm{ex})\cong\mathfrak{sl}_{m|m}(S)
:=[\mathfrak{gl}_{m|m}(S),\mathfrak{gl}_{m|m}(S)],\quad m\geqslant1,
\end{equation*}
where $\mathfrak{gl}_{m|m}(S)$ is the Lie superalgebra of $2m\times 2m$-matrices with entries in $S$.
\end{example}

\begin{proof}
It is straightforward to check that there is an isomorphism from the Lie superalgebra $\mathfrak{gl}_{m|m}(S)$ to the Lie superalgebra $\tilde{\mathfrak{p}}_m(S\oplus S^{\mathrm{op}},\mathrm{ex})$ given by:
\begin{align*}
e_{i,j}(a)&\mapsto e_{i,j}(a\oplus 0)-e_{m+j,m+i}(0\oplus a),\\
e_{i,m+j}(a)&\mapsto e_{i,m+j}(a\oplus 0)+e_{j,m+i}(0\oplus \rho(a)),\\
e_{m+i,j}(a)&\mapsto e_{m+i,j}(a\oplus 0)-e_{m+j,i}(0\oplus \rho(a)),\\
e_{m+i,m+j}(a)&\mapsto -e_{j,i}(0\oplus a)+e_{m+i,m+j}(a\oplus0),
\end{align*}
for $a\in S$ and $1\leqslant i,j\leqslant m$. Taking their derived Lie sub-superalgebras, we obtain the desired isomorphism from the Lie superalgebra $\mathfrak{sl}_{m|m}(S)$ to the Lie superalgebra $\mathfrak{p}_m(S\oplus S^{\mathrm{op}},\mathrm{ex})$.
\end{proof}
\bigskip

In order to discuss the universal central extension of the generalized periplectic Lie superalgebra $\mathfrak{p}_m(R,{}^-)$, we explore the perfectness of $\mathfrak{p}_m(R,{}^-)$. We will use the following notations:
\begin{align*}
t_{ij}(a):&=e_{ij}(a)-e_{m+j,m+i}(\bar{a}),\\
f_{ij}(a):&=e_{i,m+j}(a)+e_{j,m+i}(\rho(\bar{a})),\\
g_{ij}(a):&=e_{m+i,j}(a)-e_{m+j,i}(\rho(\bar{a})).
\end{align*}
We always denote $R_{(\pm)}:=\{a\in R|\bar{a}=\pm \rho(a)\}$.

\begin{lemma}
\label{lem:p_elm}
For $m\geqslant2$, every element $x\in\tilde{\mathfrak{p}}_m(R,{}^-)$ is written as
\begin{equation}
\begin{aligned}
x=&{t}_{11}(a)+\sum\limits_{i=2}^m({t}_{ii}(a_i)-{t}_{11}(a_i))+\sum\limits_{1\leqslant i\neq j\leqslant m}{t}_{ij}(a_{ij})\\
&+\sum\limits_{i=1}^m(e_{i,m+i}(b_i)+e_{m+i,i}(c_i))+\sum\limits_{1\leqslant i<j\leqslant m}({f}_{ij}(b_{ij})+{g}_{ij}(c_{ij})),
\end{aligned}\label{eq:p_ele}
\end{equation}
where $a, a_i, a_{ij}, b_{ij}, c_{ij}\in R$, $b_i\in R_{(+)}$ and $c_i\in R_{(-)}$ are uniquely determined by $x$. Moreover, such an element $x$ is contained in 
$\mathfrak{p}_m(R,{}^-)
=[\widetilde{\mathfrak{p}}_m(R,{}^-),\widetilde{\mathfrak{p}}_m(R,{}^-)]
$ 
if and only if $a\in [R,R]+R_{(-)}$.
\end{lemma}
\begin{proof}
The first statement follows from the definition of $\tilde{\mathfrak{p}}_m(R,{}^-)$. We show that $x\in\mathfrak{p}_m(R,{}^-)$ if and only if $a\in [R,R]+R_{(-)}$.

We observe that each term on the right hand side of (\ref{eq:p_ele}) except $t_{11}(a)$ is a super-commutator of two elements in $\widetilde{\mathfrak{p}}_m(R,{}^-)$, i.e., they are contained in $\mathfrak{p}_m(R,{}^-)=[\widetilde{p}_m(R,{}^-),\widetilde{p}_m(R,{}^-)]$.
Hence, it suffices to show that ${t}_{11}(a)\in\mathfrak{p}_m(R,{}^-)$ if and only if $a\in [R,R]+R_{(-)}$.

If $a\in[R,R]$, we write $a=\sum [a'_i,a''_i]$ with $a'_i, a''_i\in R$, then
$${t}_{11}(a)=\sum[{t}_{11}(a'_i),{t}_{11}(a''_i)]\in \mathfrak{p}_m(R,{}^-).$$
While an element $a\in R_{(-)}$ satisfies that $${t}_{11}(a)+{t}_{22}(a)={t}_{11}(a)-{t}_{22}(\rho(\bar{a}))=[{f}_{12}(a),{g}_{21}(1)]\in\mathfrak{p}_m(R,{}^-).$$
Combining with ${t}_{22}(a)-{t}_{11}(a)\in\mathfrak{p}_m(R,{}^-)$ and $\frac{1}{2}\in\Bbbk$, we conclude that ${t}_{11}(a)\in\mathfrak{p}_m(R,{}^-)$. This shows that ${t}_{11}(a)\in\mathfrak{p}_m(R,{}^-)$ if $a\in [R,R]+R_{(-)}$.

For the inverse implication, we denote 
\begin{equation}
\tau(x):=\mathrm{Tr}(A)\in R,\text{ if }x=\begin{pmatrix}A&B\\C&-\overline{A}^t\end{pmatrix}
\in\tilde{\mathfrak{p}}_m(R,{}^-).
\label{eq:tau}
\end{equation}
A straightforward computation shows that $\tau([x,y])\in[R,R]+R_{(-)}$ for all $x,y\in\tilde{\mathfrak{p}}_m(R,{}^-)$.  Hence, $a=\tau(t_{11}(a))\in [R,R]+R_{(-)}$ if ${t}_{11}(a)\in\mathfrak{p}_m(R,{}^-)$.
\end{proof}

\begin{proposition}
\label{prop:p_pft}
Let $m\geqslant2$ and $(R,{}^-)$ be a unital associative superalgebra with superinvolution.
\begin{enumerate}
\item There is an exact sequence of Lie superalgebras
$$0\rightarrow\mathfrak{p}_m(R,{}^-)\rightarrow\tilde{\mathfrak{p}}_m(R,{}^-)\rightarrow \frac{R}{[R,R]+R_{(-)}}\rightarrow0.$$
\item The Lie superalgebra $\mathfrak{p}_m(R,{}^-)$ is generated by ${t}_{ij}(a)$, ${f}_{ij}(a)$, and ${g}_{ij}(a)$ for $a\in R$ and $1\leqslant i\neq j\leqslant m$.
\item If $m\geqslant3$, then the Lie superalgebra $\mathfrak{p}_m(R,{}^-)$ is perfect, i.e.,
$$\mathfrak{p}_m(R,{}^-)=[\mathfrak{p}_m(R,{}^-),\mathfrak{p}_m(R,{}^-)].$$
\end{enumerate}
\end{proposition}
\begin{proof}
(i) Recall that $\tau:\tilde{\mathfrak{p}}_m(R,{}^-)\rightarrow R$ defined in  (\ref{eq:tau}) is a $\Bbbk$--linear map. Composing the canonical map $R\rightarrow \frac{R}{[R,R]+R_{(-)}}$, we obtain a $\Bbbk$--linear map
$$\tilde{\mathfrak{p}}_m(R,{}^-)\xrightarrow{\tau}R\rightarrow\frac{R}{[R,R]+R_{(-)}},$$
whose kernel is $\mathfrak{p}_m(R,{}^-)$ as shown in Lemma~\ref{lem:p_elm}. It yields an exact sequence of $\Bbbk$--modules:
$$0\rightarrow \mathfrak{p}_m(R,{}^-)\rightarrow\tilde{\mathfrak{p}}_m(R,{}^-)\rightarrow\frac{R}{[R,R]+R_{(-)}}\rightarrow0.$$
Note that since the Lie superalgebra $R/([R,R]+R_{(-)})$ is super-commutative, the sequence is an exact sequence of Lie superalgebras.
\medskip

(ii) By Lemma~\ref{lem:p_elm}, every element of $\mathfrak{p}(R,{}^-)$ has a decomposition as a sum of the form (\ref{eq:p_ele}). It suffices to show that each term in this decomposition is generated by $t_{ij}(a), f_{ij}(a)$ and $g_{ij}(a)$ for $a\in R$ and $1\leqslant i\neq j\leqslant m$. This follows from the following observation:
\begin{align*}
t_{11}(a)&=\frac{1}{2}[{t}_{12}(1),{t}_{21}(a)]+\frac{1}{2}[{f}_{12}(1),{g}_{21}(a)],
&&\text{for }a\in R_{(-)},\\
t_{11}([a,b])&=[{t}_{12}(a),{t}_{21}(b)]-(-1)^{|a||b|}[{t}_{12}(1),{t}_{21}(ba)],&&\text{for } a, b\in R,\\
t_{ii}(a)-t_{11}(a)&=[t_{i1}(a),t_{1i}(1)],&&\text{for }a\in R, 2\leqslant i\leqslant m,\\
e_{i,m+i}(a)&=\frac{1}{2}[t_{ij}(1),f_{ji}(a)],&&\text{for }a\in R_{(+)}, 1\leqslant i\neq j\leqslant m,\\
e_{m+i,i}(a)&=\frac{1}{2}[g_{ij}(a),t_{ji}(1)],&&\text{for }a\in R_{(-)},  1\leqslant i\neq j\leqslant m.
\end{align*}

(iii) It is known from (ii) that the Lie superalgebra $\mathfrak{p}_m(R,{}^-)$ is generated by ${t}_{ij}(a),{f}_{ij}(a),{g}_{ij}(a)$ for $a\in R$ and $1\leqslant i\neq j\leqslant m$. We shall show that
these elements are also contained in its derived Lie superalgebra $[\mathfrak{p}_m(R,{}^-),\mathfrak{p}_m(R,{}^-)]$. Note that $m\geqslant3$, for $1\leqslant i\neq j\leqslant m$, we may choose $1\leqslant k\leqslant m$ such that $i,j,k$ are pairwise distinct. Then the equalities
\begin{equation*}
{t}_{ij}(a)=[{t}_{ik}(1),{t}_{kj}(a)],\quad
{f}_{ij}(a)=[{t}_{ik}(1),{f}_{kj}(a)],\quad\text{ and }
{g}_{ij}(a)=[{g}_{ik}(a),{t}_{kj}(1)]
\end{equation*}
imply that ${t}_{ij}(a),{f}_{ij}(a),{g}_{ij}(a)\in[\mathfrak{p}_m(R,{}^-),\mathfrak{p}_m(R,{}^-)]$. Hence, $\mathfrak{p}_m(R,{}^-)$ is perfect for $m\geqslant3$.
\end{proof}

\begin{remark}
\label{rmk:p1_impft}
The Lie superalgebra $\mathfrak{p}_1(R,{}^-)$ is not necessarily perfect. For instance, if $R$ is  super-commutative, then
$$\tilde{\mathfrak{p}}_1(R,\rho):=\left\{\begin{pmatrix} a&b\\0&-a\end{pmatrix}\middle| a,b\in R\right\},\text{ and }\mathfrak{p}_1(R,\rho)=\left\{\begin{pmatrix} 0&b\\0&0\end{pmatrix}\middle| b\in R\right\}.$$
The Lie superalgebra $\mathfrak{p}_1(R,\rho)$ is not perfect since $[\mathfrak{p}_1(R,\rho),\mathfrak{p}_1(R,\rho)]=0$. In general, the condition for the perfectness of $\mathfrak{p}_1(R,{}^-)$ is unknown yet. Similarly, the Lie superalgebra $\mathfrak{p}_2(R,{}^-)$ is also not necessarily perfect. Hence, the existence of a universal central extension of $\mathfrak{p}_1(R,{}^-)$ or $\mathfrak{p}_2(R,{}^-)$ is not ensured. We only consider the second homology of $\mathfrak{p}_m(R,{}^-)$ for $m\geqslant3$.
\end{remark}

\section{Steinberg Periplectic Lie Superalgebras}
\label{sec:stp}
In the previous section, we have shown that the Lie superalgebra $\mathfrak{p}_m(R,{}^-)$ is perfect for $m\geqslant3$. The perfectness allows us to further study its universal central extension, whose kernel will finally provide us with the second homology of $\mathfrak{p}_m(R,{}^-)$.

In this section, we will introduce the notion of Steinberg periplectic Lie superalgebra $\mathfrak{stp}_m(R,{}^-)$ and prove that it is a central extension of $\mathfrak{p}_m(R,{}^-)$. Its universality will be discussed in the subsequent sections.

\begin{definition}
\label{def:stp}
Let $m\geqslant3$ and $(R,{}^-)$ be a unital associative superalgebra with superinvolution. \textit{The Steinberg periplectic Lie superalgebra }$\mathfrak{stp}_m(R,{}^-)$ is the Lie superalgebra generated by homogeneous elements $\mathbf{t}_{ij}(a)$, $\mathbf{f}_{ij}(a)$, $\mathbf{g}_{ij}(a)$ for $a\in R$ and $1\leqslant i\neq j\leqslant m$, with degrees $|a|$, $|a|+1$, $|a|+1$ respectively, and subject to the relations:
\begin{align}
&a\mapsto\mathbf{t}_{ij}(a),\mathbf{f}_{ij}(a),\mathbf{g}_{ij}(a)\text{ are all }\Bbbk\text{-linear},
&&\text{for }i\neq j,\tag{STP00}\label{STP00}\\
&\mathbf{f}_{ij}(\bar{a})=\mathbf{f}_{ji}(\rho(a)),
&&\text{for }i\neq j,\tag{STP01}\label{STP01}\\
&\mathbf{g}_{ij}(\bar{a})=-\mathbf{g}_{ji}(\rho(a)),
&&\text{for }i\neq j,\tag{STP02}\label{STP02}\\
&[\mathbf{t}_{ij}(a),\mathbf{t}_{jk}(b)]=\mathbf{t}_{ik}(ab),
&&\text{for pairwise distinct }i,j,k,\tag{STP03}\label{STP03}\\
&[\mathbf{t}_{ij}(a),\mathbf{t}_{kl}(b)]=0,
&&\text{for }i\neq j\neq k\neq l\neq i,\tag{STP04}\label{STP04}\\
&[\mathbf{t}_{ij}(a),\mathbf{f}_{jk}(b)]=\mathbf{f}_{ik}(ab),
&&\text{for pairwise distinct }i,j,k,\tag{STP05}\label{STP05}\\
&[\mathbf{t}_{ij}(a),\mathbf{f}_{kl}(b)]=0,
&&\text{for }i\neq j\neq k\neq l\neq j,\tag{STP06}\label{STP06}\\
&[\mathbf{g}_{ij}(a),\mathbf{t}_{jk}(b)]=\mathbf{g}_{ik}(ab),
&&\text{for pairwise distinct } i,j,k,\tag{STP07}\label{STP07}\\
&[\mathbf{g}_{ij}(a),\mathbf{t}_{kl}(b)]=0,
&&\text{for }l\neq k\neq j\neq i\neq k,\tag{STP08}\label{STP08}\\
&[\mathbf{f}_{ij}(a),\mathbf{f}_{kl}(b)]=0,
&&\text{for }i\neq j,\text{ and }k\neq l,\tag{STP09}\label{STP09}\\
&[\mathbf{g}_{ij}(a),\mathbf{g}_{kl}(b)]=0,
&&\text{for }i\neq j,\text{ and }k\neq l,\tag{STP10}\label{STP10}\\
&[\mathbf{f}_{ij}(a),\mathbf{g}_{jk}(b)]=\mathbf{t}_{ik}(ab),
&&\text{for pairwise distinct }i,j,k,\tag{STP11}\label{STP11}\\
&[\mathbf{f}_{ij}(a),\mathbf{g}_{kl}(b)]=0,
&&\text{for pairwise distinct }i,j,k,l,\tag{STP12}\label{STP12}
\end{align}
where $a,b\in R$ and $1\leqslant i,j,k,l\leqslant m$.
\end{definition}

As shown in Proposition~\ref{prop:p_pft} (ii), the Lie superalgebra $\mathfrak{p}_m(R,{}^-)$ is generated by $t_{ij}(a), f_{ij}(a)$ and $g_{ij}(a)$ for $a\in R$ and $1\leqslant i\neq j\leqslant m$. These generators satisfy all relations (\ref{STP00})-(\ref{STP12}). Hence, there is a canonical homomorphism of Lie superalgebras
\begin{equation}
\psi:\mathfrak{stp}_m(R,{}^-)\rightarrow\mathfrak{p}_m(R,{}^-),\label{eq:can_hom}
\end{equation}
such that 
$$\psi(\mathbf{t}_{ij}(a))=t_{ij}(a),\psi(\mathbf{f}_{ij}(a))=f_{ij}(a)\text{ and }\psi(\mathbf{g}_{ij}(a))=g_{ij}(a), \text{ for } a\in R.$$ 
Our aim in this section is to show the canonical homomorphism $\psi:\mathfrak{stp}_m(R,{}^-)\rightarrow\mathfrak{p}_m(R,{}^-)$ is a central extension. 
\bigskip

It is easy to observe that the Lie superalgebra $\mathfrak{p}_m(R,{}^-)$ has a decomposition into the direct sum of three Lie sub-superalgebras:
\begin{equation}
\mathfrak{p}_m(R,{}^-)=\mathfrak{p}_m^-(R,{}^-)\oplus\mathfrak{p}_m^0(R,{}^-)\oplus\mathfrak{p}_m^+(R,{}^-),\label{eq:p_dec}
\end{equation}
where $\mathfrak{p}_m^0(R,{}^-)$ consists of all diagonal matrices in $\mathfrak{p}_m(R,{}^-)$,
\begin{align*}
\mathfrak{p}^-(R,{}^-)&=\left\{\begin{pmatrix}A&0\\C&\overline{A}^t\end{pmatrix}\in\mathfrak{p}_m(R,{}^-)\middle|
A\text{ is strictly lower triangular and }\overline{C}^t=-\rho(C)\right\},\text{ and }\\
\mathfrak{p}^+(R,{}^-)&=\left\{\begin{pmatrix}A&B\\0&\overline{A}^t\end{pmatrix}\in\mathfrak{p}_m(R,{}^-)\middle|
A\text{ is strictly upper triangular and }\overline{B}^t=\rho(B)\right\}.
\end{align*}
We first show that the Steinberg periplectic Lie superalgebra $\mathfrak{stp}_m(R,{}^-)$ also possesses a similar decomposition. 

\begin{lemma}
\label{lem:munui_wd}
In $\mathfrak{stp}_m(R,{}^-)$, the following equalities hold:
\begin{align*}
[\mathbf{t}_{ij}(a),\mathbf{f}_{ji}(b)]=[\mathbf{t}_{ik}(a),\mathbf{f}_{ki}(b)],\text{ and }
[\mathbf{g}_{ij}(a),\mathbf{t}_{ji}(b)]=[\mathbf{g}_{ik}(a),\mathbf{t}_{ki}(b)],
\end{align*}
for $a,b\in R$ and $1\leqslant i,j,k\leqslant m$ with $i\neq j,k$.
\end{lemma}
\begin{proof}
We assume $i, j, k$ are pairwise distinct and deduce from (\ref{STP03}), (\ref{STP05}) and (\ref{STP06}) that
\begin{align*}
[\mathbf{t}_{ik}(a),\mathbf{f}_{ki}(b)]&=[[\mathbf{t}_{ij}(a),\mathbf{t}_{jk}(1)],\mathbf{f}_{ki}(b)]\\
&=[[\mathbf{t}_{ij}(a),\mathbf{f}_{ki}(b)],\mathbf{t}_{jk}(1)]+[\mathbf{t}_{ij}(a),
[\mathbf{t}_{jk}(1),\mathbf{f}_{ki}(b)]]\\
&=0+[\mathbf{t}_{ij}(a),\mathbf{f}_{ji}(b)]\\
&=[\mathbf{t}_{ij}(a),\mathbf{f}_{ji}(b)].
\end{align*}
Similarly, $[\mathbf{g}_{ij}(a),\mathbf{t}_{ji}(b)]=[\mathbf{g}_{ik}(a),\mathbf{t}_{ki}(b)]$ follows from (\ref{STP03}), (\ref{STP07}) and (\ref{STP08}).
\end{proof}

Lemma~\ref{lem:munui_wd} permits us to introduce the following well-defined elements of $\mathfrak{stp}_m(R,{}^-)$:
\begin{align*}
\mathbf{f}_{i}(a):&=[\mathbf{t}_{ij}(1),\mathbf{f}_{ji}(a)],&&\text{for some }j\neq i,\\
\mathbf{g}_{i}(a):&=[\mathbf{g}_{ij}(a),\mathbf{t}_{ji}(1)],&&\text{for some }j\neq i,\\
\mathbf{h}_{ij}(a,b):&=[\mathbf{f}_{ij}(a),\mathbf{g}_{ji}(b)],&&\text{for }i\neq j,
\end{align*}
where $a,b\in R$ and $1\leqslant i,j\leqslant m$. They satisfy
$$\mathbf{f}_i(\bar{a})=\mathbf{f}_i(\rho(a)),\text{ and }
\mathbf{g}_i(\bar{a})=-\mathbf{g}_i(\rho(a)),$$
for $a\in R$ and $1\leqslant i\leqslant m$.

\begin{proposition}
\label{prop:stp_tridec}
The Lie superalgebra $\mathfrak{stp}_m(R,{}^-)$ is decomposed into a direct sum of $\Bbbk$-modules:
\begin{equation}
\mathfrak{stp}_m(R,{}^-)=\mathfrak{stp}_m^-(R,{}^-)\oplus\mathfrak{stp}_m^0(R,{}^-)\oplus\mathfrak{stp}_m^+(R,{}^-),
\label{eq:stp_tridec}
\end{equation}
where
\begin{align*}
\mathfrak{stp}_m^0(R,{}^-):&=\mathrm{span}_{\Bbbk}\{\mathbf{h}_{ij}(a,b)|a,b\in R, 1\leqslant i\neq j\leqslant m\},\\
\mathfrak{stp}_m^+(R,{}^-):&=\mathrm{span}_{\Bbbk}\{\mathbf{t}_{ij}(a),\mathbf{f}_{ij}(a),\mathbf{f}_k(a)|a\in R, 1\leqslant i,j,k\leqslant m\text{ and }i<j\},\\
\mathfrak{stp}_m^-(R,{}^-):&=\mathrm{span}_{\Bbbk}\{\mathbf{t}_{ij}(a),\mathbf{g}_{ij}(a),\mathbf{g}_k(a)|a\in R, 1\leqslant i,j,k\leqslant m \text{ and }i>j\},
\end{align*}
are all Lie sub-superalgebras of $\mathfrak{stp}_m(R,{}^-)$ and $[\mathfrak{stp}_m^0(R,{}^-),\mathfrak{stp}_m^{\pm}(R,{}^-)]\subseteq\mathfrak{stp}_m^{\pm}(R,{}^-)$.
\end{proposition}
\begin{proof}
By (\ref{STP00})-(\ref{STP12}), it is directly verified that $\mathfrak{stp}_m^0(R,{}^-)$, $\mathfrak{stp}_m^-(R,{}^-)$ and $\mathfrak{stp}_m^+(R,{}^-)$ are all Lie sub-superalgebras of $\mathfrak{stp}_m(R,{}^-)$ and $$[\mathfrak{stp}_m^0(R,{}^-),\mathfrak{stp}_m^{\pm}(R,{}^-)]\subseteq\mathfrak{stp}_m^{\pm}(R,{}^-).$$

Let $\mathfrak{g}:=\mathfrak{stp}_m^-(R,{}^-)+\mathfrak{stp}_m^0(R,{}^-)+\mathfrak{stp}_m^+(R,{}^-)$. We show that $\mathfrak{g}=\mathfrak{stp}_m(R,{}^-)$ and the sum is direct. Indeed, it is observed from the defining relations of $\mathfrak{stp}_m(R,{}^-)$ that the $\Bbbk$--module $\mathfrak{g}$ is invariant under the adjoint action of $\mathbf{t}_{ij}(a)$, $\mathbf{f}_{ij}(a)$ and $\mathbf{g}_{ij}(a)$ for $a\in R$ and $1\leqslant i\neq j\leqslant m$, which generate the Lie superalgebra $\mathfrak{stp}_m(R,{}^-)$. In other words, $\mathfrak{g}$ is an ideal of the Lie superalgebra $\mathfrak{stp}_m(R,{}^-)$, containing a complete family of its generators. Hence, $\mathfrak{g}=\mathfrak{stp}_m(R,{}^-)$.

In order to show the sum is direct, we need to prove that the restrictions $\psi|\mathfrak{stp}_m^{\pm}(R,{}^-)$ of the canonical homomorphism (\ref{eq:can_hom}) are injective. Suppose that $\boldsymbol{x}^+\in\mathfrak{stp}_m^+(R,{}^-)$ satisfying $\psi(\boldsymbol{x}^+)=0$. Write
$$\boldsymbol{x}^+=\sum\limits_{1\leqslant i<j\leqslant m}\mathbf{t}_{ij}(a_{ij})+\sum\limits_{1\leqslant i<j\leqslant m}\mathbf{f}_{ij}(b_{ij})+\sum\limits_{i}\mathbf{f}_i(c_i),$$
where $a_{ij}, b_{ij}, c_i\in R$. Applying $\psi$ to $\boldsymbol{x}^+$, we obtain
$$0=\psi(\boldsymbol{x}^+)=\sum\limits_{1\leqslant i<j\leqslant m}({t}_{ij}(a_{ij})+{f}_{ij}(b_{ij}))
+\sum\limits_{i=1}^me_{i,m+i}(c_i+\rho(\bar{c}_i))\in\mathfrak{p}_m(R,{}^-).$$
It follows that $a_{ij}=b_{ij}=0$ for $1\leqslant i<j\leqslant m$ and $c_i+\rho(\bar{c}_i)=0$ for $i=1,\ldots,m$. Then $\mathbf{t}_{ij}(a_{ij})=\mathbf{f}_{ij}(b_{ij})=0$ for $1\leqslant i<j\leqslant m$ and
$$\mathbf{f}_i(c_i)=\frac{1}{2}\mathbf{f}_i(c_i+\rho(\bar{c}_i))=0,$$
for $i=1,\ldots,m$. i.e., $\boldsymbol{x}^+=0$.
Hence, $\psi|\mathfrak{stp}_m^+(R,{}^-)$ is injective. So is $\psi|\mathfrak{stp}_m^-(R,{}^-)$.

Now, if $0=\boldsymbol{x}^-+\boldsymbol{x}^0+\boldsymbol{x}^+$ for $\boldsymbol{x}^0\in\mathfrak{stp}_m^0(R,{}^-)$ and $\boldsymbol{x}^{\pm}\in\mathfrak{stp}_m^{\pm}(R,{}^-)$, then
$$0=\psi(\boldsymbol{x}^-)+\psi(\boldsymbol{x}^0)+\psi(\boldsymbol{x}^+)\in\mathfrak{p}_m(R,{}^-),$$
where $\psi(\boldsymbol{x}^-)$ (resp. $\psi(\boldsymbol{x}^0)$ and $\psi(\boldsymbol{x}^+)$) is contained in $\mathfrak{p}_m^-(R,{}^-)$ (resp. $\mathfrak{p}_m^0(R,{}^-)$ and $\mathfrak{p}_m^+(R,{}^-)$). The direct sum decomposition (\ref{eq:p_dec}) ensures that $\psi(\boldsymbol{x}^-)=\psi(\boldsymbol{x}^0)=\psi(\boldsymbol{x}^+)=0$. Since $\psi|\mathfrak{stp}_m^{\pm}(R,{}^-)$ are injective, we conclude that $\boldsymbol{x}^-=\boldsymbol{x}^+=0$, i.e., the sum (\ref{eq:stp_tridec}) is direct.
\end{proof}

\begin{proposition}
\label{prop:stp_ce}
Let $m\geqslant3$ and $(R,{}^-)$ be a unital associative superalgebra with superinvolution. Then $\psi:\mathfrak{stp}_m(R,{}^-)\rightarrow\mathfrak{p}_m(R,{}^-)$ is a central extension and $\ker\psi\subseteq\mathfrak{stp}_m^0(R,{}^-)$.
\end{proposition}
\begin{proof}
Let $\boldsymbol{x}\in\ker\psi$. We write $\boldsymbol{x}=\boldsymbol{x}^-+\boldsymbol{x}^0+\boldsymbol{x}^+$ with respect to the decomposition~(\ref{eq:stp_tridec}). Then
$$0=\psi(\boldsymbol{x})=\psi(\boldsymbol{x}^-)+\psi(\boldsymbol{x}^0)+\psi(\boldsymbol{x}^+)\in\mathfrak{p}_m(R,{}^-),$$
where $\psi(\boldsymbol{x}^-)$ (resp. $\psi(\boldsymbol{x}^0)$ and $\psi(\boldsymbol{x}^+)$) is contained in $\mathfrak{p}_m^-(R,{}^-)$ (resp. $\mathfrak{p}_m^0(R,{}^-)$ and $\mathfrak{p}_m^+(R,{}^-)$). It follows from the direct sum decomposition (\ref{eq:p_dec}) that $\psi(\boldsymbol{x}^-)=\psi(\boldsymbol{x}^0)=\psi(\boldsymbol{x}^+)=0$, which yields $\boldsymbol{x}^+=\boldsymbol{x}^-=0$ by the injectivity of $\psi|\mathfrak{stp}_m^{\pm}(R,{}^-)$ as shown in the proof of Proposition~\ref{prop:stp_tridec}. Hence, $\boldsymbol{x}=\boldsymbol{x}^0\in\mathfrak{stp}_m^0(R,{}^-)$.

It remains to show that every element $\boldsymbol{x}\in\ker\psi$ commutes with the generators $\mathbf{t}_{ij}(a)$, $\mathbf{f}_{ij}(a)$ and $\mathbf{g}_{ij}(a)$ for $a\in R$ and $1\leqslant i\neq j\leqslant m$. Let $\boldsymbol{x}\in\ker\psi$. We have
$$\psi([\boldsymbol{x},\mathbf{t}_{ij}(a)])=
\psi([\boldsymbol{x},\mathbf{f}_{ij}(a)])=
\psi([\boldsymbol{x},\mathbf{g}_{ij}(a)])=0.$$
According to Proposition~\ref{prop:stp_tridec}, $\boldsymbol{x}\in\ker\psi\subseteq\mathfrak{stp}_m^0(R,{}^-)$ implies that each of $[\boldsymbol{x},\mathbf{t}_{ij}(a)]$, $[\boldsymbol{x},\mathbf{f}_{ij}(a)]$ and $[\boldsymbol{x},\mathbf{g}_{ij}(a)]$ is contained in either $\mathfrak{stp}_m^+(R,{}^-)$ or $\mathfrak{stp}_m^-(R,{}^-)$. Hence, the injectivity of $\psi|\mathfrak{stp}_m^{\pm}(R,{}^-)$ ensures
$$[\boldsymbol{x},\mathbf{t}_{ij}(a)]=[\boldsymbol{x},\mathbf{f}_{ij}(a)]=[\boldsymbol{x},\mathbf{g}_{ij}(a)]=0.$$
i.e., $\boldsymbol{x}$ is a central element of $\mathfrak{stp}_m(R,{}^-)$.
\end{proof}
\bigskip

In the special case where $(R,{}^-)=(S\oplus S^{\mathrm{op}},\mathrm{ex})$ for a unital associative superalgebra $S$, Example~\ref{ex:iso_p_SS} implies that the Lie superalgebra $\mathfrak{p}_m(S\oplus S^{\mathrm{op}},\mathrm{ex})$ is isomorphic to $\mathfrak{sl}_{m|m}(S)$, whose universal central extension is \textit{the Steinberg Lie superalgebra} $\mathfrak{st}_{m|m}(S)$ for $m\geqslant3$ (c.f. \cite{ChenSun2015}). Alternatively, the Steinberg periplectic Lie superalgebra $\mathfrak{stp}_m(S\oplus S^{\mathrm{op}}, \mathrm{ex})$ is shown to be a central extension of $\mathfrak{p}_m(S\oplus S^{\mathrm{op}},\mathrm{ex})$ in Proposition~\ref{prop:stp_ce}. These two central extensions of $\mathfrak{p}_{m}(S\oplus S^{\mathrm{op}},\mathrm{ex})$ are indeed isomorphic. 

\begin{proposition}
\label{prop:stpSS}
Let $m\geqslant3$ and $S$ be an arbitrary unital associative superalgebra. Then
$$\mathfrak{stp}_m(S\oplus S^{\mathrm{op}},\mathrm{ex})\cong\mathfrak{st}_{m|m}(S)$$
as Lie superalgebras over $\Bbbk$.
\end{proposition}
\begin{proof}
The Steinberg Lie superalgebra $\mathfrak{st}_{m|m}(S)$ as defined in \cite{ChenSun2015} is the Lie superalgebra generated by homogeneous elements $\boldsymbol{e}_{ij}(a)$ of degree $|i|+|j|+|a|$ for $a\in R$ and $1\leqslant i\neq j\leqslant 2m$, subject to the relations:
\begin{align}
&a\mapsto\boldsymbol{e}_{ij}(a)\text{ is a }\Bbbk\text{-linear map},\tag{ST0}\label{ST0}\\
&[\boldsymbol{e}_{ij}(a),\boldsymbol{e}_{jk}(b)]=\boldsymbol{e}_{ik}(ab),
&&\text{ for pairwise distinct }i,j,k,\tag{ST1}\label{ST1}\\
&[\boldsymbol{e}_{ij}(a),\boldsymbol{e}_{kl}(b)]=0,
&&\text{ for }i\neq j\neq k\neq l\neq i,\tag{ST2}\label{ST2}
\end{align}
where $a,b\in R$ and $1\leqslant i,j,k,l\leqslant 2m$. 

By the defining relations, there is a homomorphism of Lie superalgebras 
$$\phi:\mathfrak{st}_{m|m}(S)\rightarrow\mathfrak{stp}_m(S\oplus S^{\mathrm{op}},\mathrm{ex})$$ 
such that
\begin{align*}
\phi(\boldsymbol{e}_{ij}(a)):&=\mathbf{t}_{ij}(a\oplus 0),
&\phi(\boldsymbol{e}_{i,m+j}(a)):&=\mathbf{f}_{ij}(a\oplus 0),\\
\phi(\boldsymbol{e}_{m+i,j}(a)):&=\mathbf{g}_{ij}(a\oplus 0),
&\phi(\boldsymbol{e}_{m+i,m+j}(a)):&=-\mathbf{t}_{ji}(0\oplus a),\\
\phi(\boldsymbol{e}_{i,m+i}(a)):&=[\mathbf{t}_{ij}(1\oplus0),\mathbf{f}_{ji}(a\oplus 0)],
&\phi(\boldsymbol{e}_{m+i,i}(a)):&=[\mathbf{g}_{ij}(a\oplus0),\mathbf{t}_{ji}(1\oplus0)],
\end{align*}
for $a\in S$ and $1\leqslant i\neq j\leqslant m$. It has the inverse $\tilde{\phi}:\mathfrak{stp}_m(S\oplus S^{\mathrm{op}},\mathrm{ex})\rightarrow\mathfrak{st}_{m|m}(S)$ given by
\begin{align*}
\tilde{\phi}(\mathbf{t}_{ij}(a\oplus b))&=\boldsymbol{e}_{ij}(a)-\boldsymbol{e}_{m+j,m+i}(b),\\
\tilde{\phi}(\mathbf{f}_{ij}(a\oplus b))&=\boldsymbol{e}_{i,m+j}(a)+\boldsymbol{e}_{j,m+i}(\rho(b)),\\
\tilde{\phi}(\mathbf{g}_{ij}(a\oplus b)&=\boldsymbol{e}_{m+i,j}(a)-\boldsymbol{e}_{m+j,i}(\rho(b)),
\end{align*}
for $a,b\in S$ and $1\leqslant i\neq j\leqslant m$. Hence, we obtain the desired isomorphism.
\end{proof}

\section{Kernels of The Central Extensions}
\label{sec:kernel}
This section is devoted to characterizing the kernel of the canonical homomorphism 
$$\psi:\mathfrak{stp}_m(R,{}^-)\rightarrow\mathfrak{p}_m(R,{}^-),$$
which is demonstrated to be a central extension in the previous section. The characterization of the kernel needs the notion of the first $\mathbb{Z}/2\mathbb{Z}$-graded dihedral homology $\fourIdx{}{+}{}{1}{\mathrm{HD}}(R,{}^-)$ of a unital associative superalgebra $(R,{}^-)$ with superinvolution.

The $\mathbb{Z}/2\mathbb{Z}$-graded dihedral homology of $(R,{}^-)$ is a natural $\mathbb{Z}/2\mathbb{Z}$-graded analogue of the dihedral homology of a unital associative algebra with anti-involution. It can be defined as the homology of the coinvariant Hochschild complex under the action of the dihedral group as Definition~5.2.7 in \cite{Loday1998}. For the use in this paper, we only describe its degree one term here:

Let $I$ be the $\Bbbk$-submodule of $R\otimes_{\Bbbk}R$ spanned by
\begin{equation*}
a\otimes b+(-1)^{|a||b|}b\otimes a, \quad
a\otimes b+\bar{a}\otimes\bar{b},\text{ and }
(-1)^{|a||c|}ab\otimes c+(-1)^{|b||a|}bc\otimes a+(-1)^{|c||b|}ca\otimes b,
\end{equation*}
for $a,b,c\in R$. Let $\langle R,R\rangle:=(R\otimes_{\Bbbk}R)/I$ and $\langle a,b\rangle=a\otimes b+I$, then \textit{the first $\mathbb{Z}/2\mathbb{Z}$-graded dihedral homology of $(R,{}^-)$} is 
\begin{equation}
\fourIdx{}{+}{}{1}{\mathrm{HD}}(R,{}^-):=\left\{\sum\limits_{i}\langle a_i,b_i\rangle\middle|\sum\limits_{i}\overline{[a_i,b_i]}=-\sum\limits_{i}[a_i,b_i] \right\}.
\label{eq:HD1}
\end{equation}

\begin{proposition}
\label{prop:stp_ct}
Let $m\geqslant3$ and $(R,{}^-)$ be a unital associative superalgebra with superinvolution. Then there is an isomorphism of $\Bbbk$-modules
\begin{equation*}
\ker\psi\cong\fourIdx{}{+}{}{1}{\mathrm{HD}}(R,{}^-\circ\rho),
\end{equation*}
where $\rho$ is the $\Bbbk$-linear map as in (\ref{eq:rho}) and ${}^-\circ\rho$, the composition of the superinvolution ${}^-$ with $\rho$, is also a superinvolution on $R$.
\end{proposition}

In order to prove this proposition, we need a few lemmas:

\begin{lemma}
\label{lem:ht_wd}
The elements $\mathbf{h}_{ij}(a,b)=[\mathbf{f}_{ij}(a),\mathbf{g}_{ji}(b)]\in\mathfrak{stp}_m(R,{}^-)$ satisfy
\begin{align}
\mathbf{h}_{1i}(a,b)-(-1)^{|a||b|}\mathbf{h}_{1i}(1,ba)
&=\mathbf{h}_{1k}(a,b)-(-1)^{|a||b|}\mathbf{h}_{1k}(1,ba),
\label{eq:h_wd}\\
\mathbf{h}_{i1}(1,a)+\mathbf{h}_{1j}(1,a)-\mathbf{h}_{ij}(1,a)
&=\mathbf{h}_{k1}(1,a)+\mathbf{h}_{1l}(1,a)-\mathbf{h}_{kl}(1,a),
\label{eq:t_wd}
\end{align}
for $a,b\in R$ and $2\leqslant i,j,k,l\leqslant m$ with $i\neq j$ and $k\neq l$.
\end{lemma}
\begin{proof}
We first observe that (\ref{STP11}) implies
\begin{equation}
[\mathbf{t}_{ij}(a),\mathbf{t}_{ji}(b)]
=\mathbf{h}_{ik}(a,b)
-(-1)^{|a||b|}\mathbf{h}_{jk}(1,ba),
\label{eq:stp_tth}
\end{equation}
for $a,b\in R$ and pairwise distinct $i,j,k$. 

For (\ref{eq:h_wd}), we deduce from (\ref{eq:stp_tth}) that 
$$\mathbf{h}_{1i}(a,b)-(-1)^{|a||b|}\mathbf{h}_{1i}(1,ba)
=[\mathbf{t}_{1k}(a),\mathbf{t}_{k1}(b)]-(-1)^{|a||b|}[\mathbf{t}_{1k}(1),\mathbf{t}_{k1}(ba)]$$
for $a,b\in R$ and $2\leqslant i\neq k\leqslant m$, whose right-hand side is independent of $i$. Hence, (\ref{eq:h_wd}) follows.

For (\ref{eq:t_wd}), it obviously holds if $i=k$ since both sides of the equality are equal to $[\mathbf{t}_{1i}(1),\mathbf{t}_{i1}(a)]$. If $2\leqslant i\neq k\leqslant m$, the Jacobi identity implies that
\begin{align*}
[\mathbf{t}_{1i}(1),\mathbf{t}_{i1}(a)]
&=[[\mathbf{t}_{1k}(1),\mathbf{t}_{ki}(1)],\mathbf{t}_{i1}(a)]
=[\mathbf{t}_{ki}(1),\mathbf{t}_{ik}(a)]+[\mathbf{t}_{1k}(1),\mathbf{t}_{k1}(a)],
\end{align*}
which yields (\ref{eq:t_wd}) by applying (\ref{eq:stp_tth}). 
\end{proof}

Lemma~\ref{lem:ht_wd} ensures that
\begin{equation*}
\boldsymbol{\lambda}(a,b):=\mathbf{h}_{1i}(a,b)-(-1)^{|a||b|}\mathbf{h}_{1i}(1,ba)\in\mathfrak{stp}_m(R,{}^-)
\end{equation*}
is independent of $2\leqslant i\leqslant m$, and
\begin{equation*}
\boldsymbol{\mu}(a):=\mathbf{h}_{i1}(1,a)+\mathbf{h}_{1j}(1,a)-\mathbf{h}_{ij}(1,a)\in\mathfrak{stp}_m(R,{}^-)
\end{equation*}
is independent of $2\leqslant i\neq j\leqslant m$. Moreover, they satisfy the following properties:

\begin{lemma}
\label{lem:ht_rln}
For $a,b,c\in R$, we have
\begin{enumerate}
\item $(-1)^{|a||c|}\boldsymbol{\lambda}(ab,c)+(-1)^{|b||a|}\boldsymbol{\lambda}(bc,a)+(-1)^{|c||b|}\boldsymbol{\lambda}(ca,b)=0$,
\item $\boldsymbol{\lambda}(a,1)=\boldsymbol{\lambda}(1,b)=0$,
\item $\boldsymbol{\lambda}(a,b)=-(-1)^{|a||b|}\boldsymbol{\lambda}(b,a)$,
\item $\boldsymbol{\mu}(\bar{a})=-\boldsymbol{\mu}(\rho(a))$.
\end{enumerate}
\end{lemma}
\begin{proof}
(i) We observe from (\ref{eq:stp_tth}) that
\begin{equation*}
\boldsymbol{\lambda}(ab,c)=[\mathbf{t}_{1j}(ab),\mathbf{t}_{j1}(c)]
-(-1)^{|a||c|+|b||c|}[\mathbf{t}_{1j}(1),\mathbf{t}_{j1}(cab)],
\end{equation*}
for $a,b,c\in R$ and $j\neq 1$. On one hand, $\mathbf{t}_{j1}(cab)=[\mathbf{t}_{ji}(ca),\mathbf{t}_{i1}(b)]$ for $2\leqslant i\neq j\leqslant m$, then the Jacobi identity implies that
$$[\mathbf{t}_{1j}(1),\mathbf{t}_{j1}(cba)]
=[\mathbf{t}_{1i}(ca),\mathbf{t}_{i1}(b)]-[\mathbf{t}_{ji}(ca),\mathbf{t}_{ij}(b)].$$
It follows that
\begin{equation}
\boldsymbol{\lambda}(ab,c)
=[\mathbf{t}_{1j}(ab),\mathbf{t}_{j1}(c)]
-(-1)^{|a||c|+|b||c|}[\mathbf{t}_{1i}(ca),\mathbf{t}_{i1}(b)]
+(-1)^{|a||c|+|b||c|}[\mathbf{t}_{ji}(ca),\mathbf{t}_{ij}(b)].
\label{eq:lambdatt1}
\end{equation}

On the other hand, $\mathbf{t}_{j1}(cab)=[\mathbf{t}_{ji}(c),\mathbf{t}_{i1}(ab)]$, then
$$[\mathbf{t}_{1j}(1),\mathbf{t}_{j1}(cab)]
=[\mathbf{t}_{1i}(c),\mathbf{t}_{i1}(ab)]
-[\mathbf{t}_{ji}(c),\mathbf{t}_{ij}(ab)],$$
which yields that
\begin{equation}
\boldsymbol{\lambda}(ab,c)
=[\mathbf{t}_{1j}(ab),\mathbf{t}_{j1}(c)]
+[\mathbf{t}_{i1}(ab),\mathbf{t}_{1i}(c)]
-[\mathbf{t}_{ij}(ab),\mathbf{t}_{ji}(c)].
\label{eq:lambdatt2}
\end{equation}

Note that (\ref{STP03}) and the Jacobi identity imply that
\begin{equation*}
(-1)^{|a||c|}[\mathbf{t}_{ij}(ab),\mathbf{t}_{ji}(c)]
+(-1)^{|a||b|}[\mathbf{t}_{ki}(bc),\mathbf{t}_{ik}(a)]
+(-1)^{|b||c|}[\mathbf{t}_{jk}(ca),\mathbf{t}_{kj}(b)]
=0,
\end{equation*}
for $a,b,c\in R$ and pairwise distinct $i,j,k$, we obtain (i) by applying (\ref{eq:lambdatt1}) to $\boldsymbol{\lambda}(ab,c)$, $\boldsymbol{\lambda}(bc,a)$ and applying (\ref{eq:lambdatt2}) to $\boldsymbol{\lambda}(ca,b)$. 
\medskip

(ii) $\boldsymbol{\lambda}(1,b)=0$ is obvious. Taking $b=c=1$ in (i), we obtain
$$\boldsymbol{\lambda}(a,1)+\boldsymbol{\lambda}(1,a)+\boldsymbol{\lambda}(a,1)=0$$
which implies $\boldsymbol{\lambda}(a,1)=0$ since $\frac{1}{2}\in\Bbbk$.
\medskip

(iii) follows from (i) by taking $c=1$.
\medskip

(iv) follows from the equality $\mathbf{h}_{ij}(a,b)=-\mathbf{h}_{ji}(\rho(\bar{a}),\rho(\bar{b}))$.
\end{proof}

\begin{lemma}
\label{lem:diag}
Every element $x\in\mathfrak{stp}_{m}^0(R,{}^-)$ can be written as
\begin{equation}
x=\sum\limits_{i\in I_x}\boldsymbol{\lambda}(a_i,b_i)+\boldsymbol{\mu}(c)+\sum\limits_{j=2}^m\mathbf{h}_{1j}(1,d_j),
\label{eq:diag_ele}
\end{equation}
where $I_x$ is a finite index set, $a_i, b_i, c, d_j\in R$ for $i\in I_x$. Moreover,
\begin{equation}
\boldsymbol{\mu}([a,b])=\boldsymbol{\lambda}(a,b)+\boldsymbol{\lambda}(\rho(\bar{a}),\rho(\bar{b})),
\label{eq:t_h}
\end{equation}
for $a,b\in R$.
\end{lemma}
\begin{proof}
Recall that $\mathfrak{stp}_{m}^0(R,{}^-)$ is spanned by $\mathbf{h}_{ij}(a,b)$ for $a,b\in R$ and $1\leqslant i\neq j\leqslant m$. It suffices to show that $\mathbf{h}_{ij}(a,b)$ can be written in the form of (\ref{eq:diag_ele}).

If $i=1$ and $2\leqslant j\leqslant m$, we have
$$-\mathbf{h}_{j1}(\rho(\bar{a}),\rho(\bar{b}))=\mathbf{h}_{1j}(a,b)=\boldsymbol{\lambda}(a,b)+(-1)^{|a||b|}\mathbf{h}_{1j}(1,ba).$$
For $2\leqslant i\neq j\leqslant m$, since $\mathbf{f}_{ij}(a)=[\mathbf{t}_{i1}(1),\mathbf{f}_{1j}(a)]$, the Jacobi identity implies that 
$$\mathbf{h}_{ij}(a,b)=[\mathbf{f}_{ij}(a),\mathbf{g}_{ji}(b)]=[\mathbf{f}_{1j}(a),\mathbf{g}_{j1}(b)]+[\mathbf{t}_{i1}(1),\mathbf{t}_{1i}(ab)]=\mathbf{h}_{1j}(a,b)+[\mathbf{t}_{i1}(1),\mathbf{t}_{1i}(ab)].$$
Combining (\ref{eq:stp_tth}), we obtain
\begin{equation}
\mathbf{h}_{ij}(a,b)=\mathbf{h}_{1j}(a,b)+\mathbf{h}_{ij}(1,ab)-\mathbf{h}_{1j}(1,ab)=\mathbf{h}_{1j}(a,b)-\boldsymbol{\mu}(ab)+\mathbf{h}_{i1}(1,ab),
\label{eq:mu_h}
\end{equation}
which is of the form (\ref{eq:diag_ele}) since $\mathbf{h}_{1j}(a,b)$ and $\mathbf{h}_{i1}(1,ab)$ are also of this form.
\medskip

Now, we prove (\ref{eq:t_h}). Fix $2\leqslant i\neq j\leqslant m$, it follows from (\ref{eq:mu_h}) and Lemma~\ref{lem:ht_rln} that
\begin{align*}
\boldsymbol{\mu}(ab)&=\mathbf{h}_{1j}(a,b)+\mathbf{h}_{i1}(1,ab)-\mathbf{h}_{ij}(a,b),\\
-(-1)^{|a||b|}\boldsymbol{\mu}(ba)
&=\boldsymbol{\mu}(\rho(\bar{a})\rho(\bar{b}))
=\mathbf{h}_{1i}(\rho(\bar{a}),\rho(\bar{b}))+\mathbf{h}_{j1}(1,\rho(\bar{a})\rho(\bar{b}))-\mathbf{h}_{ji}(\rho(\bar{a}),\rho(\bar{b}))).
\end{align*}
On the other hand, (\ref{STP01}) and (\ref{STP02}) yield
$$\mathbf{h}_{ij}(a,b)=[\mathbf{f}_{ij}(a),\mathbf{g}_{ji}(b)]
=-[\mathbf{f}_{ji}(\rho(\bar{a})),\mathbf{g}_{ij}(\rho(\bar{b}))]
=-\mathbf{h}_{ji}(\rho(\bar{a}),\rho(\bar{b})).$$
Hence,
\begin{align*}
\boldsymbol{\mu}(ab)&=\mathbf{h}_{1j}(a,b)-(-1)^{|a||b|}\mathbf{h}_{1i}(1,\rho(\bar{a})\rho(\bar{b}))-\mathbf{h}_{ij}(a,b),\\
-(-1)^{|a||b|}\boldsymbol{\mu}(ba)
&=\mathbf{h}_{1i}(\rho(\bar{a}),\rho(\bar{b}))-(-1)^{|a||b|}\mathbf{h}_{1j}(1,ba)+\mathbf{h}_{ij}(a,b)).
\end{align*}
Therefore, $\boldsymbol{\mu}([a,b])=\boldsymbol{\mu}(ab)-(-1)^{|a||b|}\boldsymbol{\mu}(ba)=\boldsymbol{\lambda}(a,b)+\boldsymbol{\lambda}(\rho(\bar{a}),\rho(\bar{b}))$.
\end{proof}

Now, we may proceed to prove Proposition~\ref{prop:stp_ct}.

\begin{proof}[Proof of Proposition~\ref{prop:stp_ct}]
We shall explicitly construct an isomorphism from $\fourIdx{}{+}{}{1}{\mathrm{HD}}(R,{}^-\circ\rho)$ to $\ker\psi$. Recall from (\ref{eq:HD1}) that $\fourIdx{}{+}{}{1}{\mathrm{HD}}(R,{}^-\circ\rho)$ is a $\Bbbk$-submodule of $\langle R,R\rangle$ that is the quotient $(R\otimes_{\Bbbk}R)/I$, where $I$ is the $\Bbbk$-submodule of $R\otimes_{\Bbbk}R$ spanned by 
\begin{eqnarray*}
&a\otimes b+(-1)^{|a||b|}b\otimes a,&\\
&a\otimes b+\rho(\bar{a})\otimes\rho(\bar{b}),\text{ and }&\\
&(-1)^{|a||c|}ab\otimes c+(-1)^{|b||a|}bc\otimes a+(-1)^{|c||b|}ca\otimes b&
\end{eqnarray*} 
for $a,b,c\in R$. By Lemmas~\ref{lem:ht_rln} and~\ref{lem:diag}, there is a well-defined $\Bbbk$-linear map:
\begin{align*}
\eta:\langle R,R\rangle&\rightarrow\mathfrak{stp}_m(R,{}^-),\\
\langle a,b\rangle&\mapsto\boldsymbol{\lambda}(a,b)-\frac{1}{2}\boldsymbol{\mu}([a,b])
=\frac{1}{2}(\boldsymbol{\lambda}(a,b)-\boldsymbol{\lambda}(\rho(\bar{a}),\rho(\bar{b}))).
\end{align*}
Our task is to prove the restriction of $\eta$ to $\fourIdx{}{+}{}{1}{\mathrm{HD}}(R,{}^-\circ\rho)$ is an isomorphism onto $\ker\psi$.
\medskip

Firstly, we verify that $\eta\left(\fourIdx{}{+}{}{1}{\mathrm{HD}}(R,{}^-\circ\rho)\right)\subseteq\ker\psi$. For $\sum_i\langle a_i,b_i\rangle\in\fourIdx{}{+}{}{1}{\mathrm{HD}}(R,{}^-\circ\rho)$, we have
$$\psi(\eta(\textstyle{\sum}_i\langle a_i,b_i\rangle))
=\frac{1}{2}\textstyle{\sum}_i\psi(\boldsymbol{\lambda}(a_i,b_i)-\boldsymbol{\lambda}(\rho(\bar{a}_i),\rho(\bar{b}_i)))
=\frac{1}{2}\textstyle{\sum}_ie_{11}([a_i,b_i]-[\rho(\bar{a}_i),\rho(\bar{b}_i)])
=0,$$
since $\sum_i\overline{[a_i,b_i]}=-\sum_i[\rho(a_i),\rho(b_i)]$. 
\medskip

Secondly, we show that $\eta$ maps $\fourIdx{}{+}{}{1}{\mathrm{HD}}(R,{}^-\circ\rho)$ onto $\ker\psi$. Let $\boldsymbol{x}\in\ker{\psi}$.  By Proposition~\ref{prop:stp_ce}, $\boldsymbol{x}\in\mathfrak{stp}_m^0(R,{}^-)$ and we may write as in Lemma~\ref{lem:diag} that
$$\boldsymbol{x}=\sum_{i\in I_{\boldsymbol{x}}}\boldsymbol{\lambda}(a_i,b_i)+\boldsymbol{\mu}(c)+\sum\limits_{j=2}^m\mathbf{h}_{1j}(1,d_j),$$
where $a_i,b_i,c,d_j\in R$ for $i\in I_{\boldsymbol{x}}$ (a finite set). Then $\psi(\boldsymbol{x})=0$ implies that $d_j=0$ for $j=2,\ldots,m$ and
$$\sum\limits_{i\in I_{\boldsymbol{x}}}[a_i,b_i]=-c_{(-)}\in R_{(-)}.$$
Hence, $\sum_i\langle a_i,b_i\rangle\in\fourIdx{}{+}{}{1}{\mathrm{HD}}(R,{}^-\circ\rho)$. Moreover, it is known from Lemma~\ref{lem:ht_rln} that $\boldsymbol{\mu}(\bar{a})=-\boldsymbol{\mu}(\rho(a))$ for  $a\in R$. Hence,
$$\boldsymbol{\mu}(c)=\frac{1}{2}\boldsymbol{\mu}(c_{(-)})=-\frac{1}{2}\sum_{i\in I_{\boldsymbol{x}}}\boldsymbol{\mu}([a_i,b_i]),$$
which yields that
$$\boldsymbol{x}=\sum_{i\in I_{\boldsymbol{x}}}(\boldsymbol{\lambda}(a_i,b_i)-\frac{1}{2}\boldsymbol{\mu}([a_i,b_i]))=\sum_{i\in I_x}\eta(\langle a_i,b_i\rangle).$$
\medskip

Finally, we prove the injectivity of $\eta$. It is known that $\eta$ is injective if $\phi\circ\eta$ is injective for another map $\phi$. Hence, the trick of the proof is to define a Lie superalgebra structure on the $\Bbbk$-module $\mathfrak{p}_m(R,{}^-)\oplus\langle R,R\rangle$ and a homomorphism $\phi:\mathfrak{stp}_m(R,{}^-)\rightarrow\mathfrak{p}_m(R,{}^-)\oplus\langle R,R\rangle$ such that $\phi\circ\eta$ is injective.

In order to define a Lie superalgebra structure on $\mathfrak{p}_m(R,{}^-)\oplus\langle R,R\rangle$, we employ a $2$-cocycle on $\mathfrak{p}_m(R,{}^-)$ with values in $\langle R,R\rangle$. Indeed, it is known from \cite{ChenSun2015} that there is a $2$-cocycle $\alpha:\mathfrak{gl}_{m|m}(R)\times\mathfrak{gl}_{m|m}(R)\rightarrow\langle R, R\rangle$ defined by
$$\alpha(e_{ij}(a),e_{kl}(b))=\delta_{jk}\delta_{il}(-1)^{|i|(|i|+|a|+|b|)}\langle a,b\rangle,$$
for $a,b\in R$ and $1\leqslant i,j,k,l\leqslant 2m$. Then the restriction of $\alpha$ to $\mathfrak{p}_m(R,{}^-)\times\mathfrak{p}_m(R,{}^-)$ is a $2$-cocycle on $\mathfrak{p}_m(R,{}^-)\subseteq\mathfrak{gl}_{m|m}(R)$. Hence, $\mathfrak{p}_m(R,{}^-)\oplus\langle R,R\rangle$ is a Lie superalgebra under:
$$[x\oplus c, y\oplus c']=[x,y]\oplus \alpha(x,y),\quad x,y\in\mathfrak{p}_m(R,{}^-)\text{ and }c,c'\in\langle R,R\rangle.$$

The existence of $\phi:\mathfrak{stp}_m(R,{}^-)\rightarrow\mathfrak{p}_m(R,{}^-)\oplus\langle R,R\rangle$ follows from a case-by-case verification that ${t}_{ij}(a)\oplus0$, ${f}_{ij}(a)\oplus0$ and ${g}_{ij}(a)\oplus0\in\mathfrak{p}_m(R,{}^-)\oplus\langle R,R\rangle$ satisfy all relations (\ref{STP00})-(\ref{STP12}). Such a homomorphism $\phi$ satisfies
$$\phi(\mathbf{t}_{ij}(a))={t}_{ij}(a)\oplus0,
\quad\phi(\mathbf{f}_{ij}(a))={f}_{ij}(a)\oplus0,
\quad\phi(\mathbf{g}_{ij}(a))={g}_{ij}(a)\oplus0,$$
for $a\in R$ and $1\leqslant i\neq j\leqslant m$. In particular,
$$
\phi(\mathbf{h}_{ij}(a,b))
=\phi([\mathbf{f}_{ij}(a),\mathbf{g}_{ji}(b)])=[{f}_{ij}(a)\oplus0,{g}_{ji}(b)\oplus0]
=({t}_{ii}(ab)-{t}_{jj}(\rho(\bar{a})\rho(\bar{b})))\oplus2\langle a,b\rangle.$$
Now, $\langle1,a\rangle=0$ since $\langle1,a\rangle=-\langle a,1\rangle$ and $\langle a,1\rangle+\langle a,1\rangle+\langle 1,a\rangle=0$. Hence, 
$$\phi(\boldsymbol{\lambda}(a,b))=\phi(\mathbf{h}_{1i}(a,b))-(-1)^{|a||b|}\phi(\mathbf{h}_{1i}(1,ba))={t}_{11}([a,b])\oplus2\langle a,b\rangle,$$
which further implies that 
$$\phi(\eta(\langle a,b\rangle))=\frac{1}{2}(\phi(\boldsymbol{\lambda}(a,b))-\phi(\boldsymbol{\lambda}(\rho(\bar{a}),\rho(\bar{b}))))\\
=\frac{1}{2}{t}_{11}([a,b]-[\rho(\bar{a}),\rho(\bar{b})])\oplus2\langle a,b\rangle.$$
Therefore, $\phi\circ\eta$ is injective, so is $\eta$.
\end{proof}

\section{Universality of The Central Extensions}
\label{sec:uce}
 
In Section~\ref{sec:stp}, we created the central extension $\psi:\mathfrak{stp}_m(R,{}^-)\rightarrow\mathfrak{p}_m(R,{}^-)$, whose kernel was identified with $\fourIdx{}{+}{}{1}{\mathrm{H}}(R,{}^-\circ\rho)$ as in Section~\ref{sec:kernel}. We are now in the position to demonstrate the universality of this central extension for $m\geqslant5$. The most crucial part of the proof is the existence of a homomorphism from $\mathfrak{stp}_m(R,{}^-)$  
to an arbitrary central extension $\mathfrak{E}$ of $\mathfrak{p}_m(R,{}^-)$. These homomorphisms will be constructed via finding certain elements of $\mathfrak{E}$ such that the defining relations of $\mathfrak{stp}_m(R,{}^-)$ are satisfied.

Let $\varphi:\mathfrak{E}\rightarrow\mathfrak{p}_m(R,{}^-)$ be an arbitrary central extension of $\mathfrak{p}_m(R,{}^-)$ with $m\geqslant3$. For $a\in R$ and $1\leqslant i\neq j\leqslant m$, we pick
$$\hat{t}_{ij}(a)\in\varphi^{-1}({t}_{ij}(a)),
\quad\hat{f}_{ij}(a)\in\varphi^{-1}({f}_{ij}(a)),\text{ and }\hat{g}_{ij}(a)\in\varphi^{-1}({g}_{ij}(a)).$$
Obviously, the element  $[\hat{x},\hat{y}]\in\mathfrak{E}$ is independent of the choice of representatives $\hat{x}\in\varphi^{-1}(x)$ and $\hat{y}\in\varphi^{-1}(y)$ for $x,y\in\mathfrak{p}_m(R,{}^-)$.  Moreover, we have the following lemma\textcolor{red}{:}

\begin{lemma}
\label{lem:ele_tilde}
In the Lie superalgebra $\mathfrak{E}$, the following equalities hold:
\begin{enumerate}
\item $[\hat{f}_{ik}(a),\hat{g}_{kj}(b)]=[\hat{f}_{il}(a),\hat{g}_{lj}(b)]$,
\item $[\hat{t}_{ik}(a),\hat{f}_{kj}(b)]=[\hat{t}_{il}(a),\hat{f}_{lj}(b)]$,
\item $[\hat{g}_{ik}(a),\hat{t}_{kj}(b)]=[\hat{g}_{il}(a),\hat{t}_{lj}(b)]$,
\end{enumerate}
for $a,b\in R$ and pairwise distinct $i,j,k,l$.
\end{lemma}
\begin{proof}
(i) Since $[\hat{t}_{lk}(1),\hat{f}_{ki}(\rho(\bar{a}))]=\hat{f}_{li}(\rho(\bar{a}))\pmod{\ker\varphi}
$ and $\hat{f}_{il}(a)=\hat{f}_{li}(\rho(\bar{a}))\pmod{\ker\varphi}
$, we have $[\hat{f}_{ik}(a),\hat{t}_{lk}(1)]=-\hat{f}_{il}(a)\pmod{\ker\varphi}
$. Hence,
\begin{align*}
[\hat{f}_{il}(a),\hat{g}_{lj}(b)]
&=[[\hat{f}_{ik}(a),\hat{t}_{kl}(1)],\hat{g}_{lj}(b)]\\
&=[[\hat{f}_{ik}(a),\hat{g}_{lj}(b)],\hat{t}_{lk}(1)]+[\hat{f}_{ik}(a),[\hat{t}_{kl}(1),\hat{g}_{lj}(b)]]\\
&=0+[\hat{f}_{ik}(a),\hat{g}_{kj}(b)],
\end{align*}
which shows (i). The equalities (ii) and (iii) follow similarly.
\end{proof}

According to Lemma~\ref{lem:ele_tilde}, we define:
\begin{equation}
\tilde{t}_{ij}(a):=[\hat{f}_{ik}(a),\hat{g}_{kj}(1)],\quad
\tilde{f}_{ij}(a):=[\hat{t}_{ik}(1),\hat{f}_{kj}(a)],\text{ and }
\tilde{g}_{ij}(a):=[\hat{g}_{ik}(a),\hat{t}_{kj}(1)],\label{eq:tilde}
\end{equation}
where $a\in R$, $1\leqslant i\neq j\leqslant m$, and $1\leqslant k\leqslant m$ is chosen such that $k\neq i,j$.

\begin{lemma}
\label{lem:tilde_rln}
Suppose $m\geqslant3$. The elements $\tilde{t}_{ij}(a),\tilde{f}_{ij}(a)$ and $\tilde{g}_{ij}(a)$ for $a\in R$ and $1\leqslant i\neq j\leqslant m$ as defined in (\ref{eq:tilde}) satisfy all relations (\ref{STP00})-(\ref{STP12}) except (\ref{STP04}) and (\ref{STP10}). Moreover, for $a,b\in R$, we have
\begin{align}
\/[\tilde{t}_{ik}(a),\tilde{t}_{jk}(b)]&=0,&&\text{if }i,j,k\text{ are pairwise distinct,}\tag{STP04a}\label{STP04a}\\
\/[\tilde{t}_{ij}(a),\tilde{t}_{kl}(b)]&=0,&&\text{if }i,j,k,l\text{ are pairwise distinct,}\tag{STP04b}\label{STP04b}\\
\/[\tilde{g}_{ij}(a),\tilde{g}_{ij}(b)]&=0,&&\text{if }i\neq j,\tag{STP10a}\label{STP10a}\\
\/[\tilde{g}_{ij}(a),\tilde{g}_{jk}(b)]&=0,&&\text{if }i,j,k\text{ are pairwise distinct.}\tag{STP10b}\label{STP10b}
\end{align}
\end{lemma}

\begin{proof}
The proof is accomplished by verifying the relations one-by-one. 

For (\ref{STP00}), the map $a\mapsto\tilde{t}_{ij}(a)$ is obviously $\Bbbk$-linear since both $\hat{f}_{ik}(ca)$ and $c\hat{f}_{ik}(a)$ are contained in $\varphi^{-1}(f_{ik}(ca))$ for $a\in R$ and $c\in\Bbbk$, and so are the maps $a\mapsto\tilde{f}_{ij}(a)$ and $a\mapsto\tilde{g}_{ij}(a)$.

For (\ref{STP01}), choose $1\leqslant k\leqslant m$ such that $i,j,k$ are pairwise distinct and set
$\tilde{h}_{ik}=[\hat{f}_{ik}(1),\hat{g}_{ki}(1)].$ We first show that 
\begin{equation}
[\tilde{h}_{ik},\tilde{f}_{ij}(a)]=\tilde{f}_{ij}(a).\label{eq:tildehf}
\end{equation}
By the definition of $\tilde{f}_{ij}(a)$, we have
\begin{align*}
[\tilde{h}_{ik},\tilde{f}_{ij}(a)]=[\tilde{h}_{ik},[\hat{t}_{ik}(1),\hat{f}_{kj}(a)]]
=[[\tilde{h}_{ik},\hat{t}_{ik}(1)],\tilde{f}_{kj}(a)]+[\hat{t}_{ik}(1),[\tilde{h}_{ik},\hat{f}_{kj}(a)]].
\end{align*}
On the other hand, $\varphi(\tilde{h}_{ik})=e_{ii}(1)-e_{kk}(1)-e_{m+i,m+i}(1)+e_{m+k,m+k}(1), \varphi(\hat{t}_{ik}(1))=t_{ik}(1)$ and $\varphi(\hat{f}_{kj}(a))=f_{kj}(a)$, which yield that
$$\varphi([\tilde{h}_{ik},\hat{t}_{ik}(1)])= 2t_{ik}(1),\text{ and }\varphi([\tilde{h}_{ik},\hat{f}_{kj}(a)])=-f_{kj}(a).$$
Since $[\hat{x},\hat{y}]$ is independent of the choice of representatives $\hat{x}\in\varphi^{-1}(x)$ and $\hat{y}\in\varphi^{-1}(y)$, we have
$$[[\tilde{h}_{ik},\hat{t}_{ik}(1)],\tilde{f}_{kj}(a)]=2[\hat{t}_{ik}(1),\hat{f}_{kj}(a)],\text{ and }[\hat{t}_{ik}(1),[\tilde{h}_{ik},\hat{f}_{kj}(a)]]=-[\hat{t}_{ik}(1),\hat{f}_{kj}(a)],$$
which imply (\ref{eq:tildehf}).

Now, $f_{ij}(\bar{a})=f_{ji}(\rho(a))\in\mathfrak{p}_m(R,{}^-)$ implies $\tilde{f}_{ij}(\bar{a})=\tilde{f}_{ji}(\rho(a))\pmod{\ker\varphi}
$. Hence,
$$0=[\tilde{h}_{ik},\tilde{f}_{ij}(\bar{a})-\tilde{f}_{ji}(\rho(a))]=\tilde{f}_{ij}(\bar{a})-\tilde{f}_{ji}(\rho(a)).$$
The relation (\ref{STP01}) is verified.

Other relations can be checked by similar arguments. We omit tedious discussion here.
\end{proof}

\begin{remark}The relation (\ref{STP04}) states that
$$[\mathbf{t}_{ij}(a),\mathbf{t}_{kl}(b)]=0,\text{ if }i\neq j\neq k\neq l\neq i,$$
which is equivalent to (\ref{STP04a}), (\ref{STP04b}) and
\begin{equation}
[\mathbf{t}_{ij}(a),\mathbf{t}_{ik}(b)]=0,\text{ if }i,j,k\text{ are pairwise distinct.}\tag{STP04c}\label{STP04c}
\end{equation}
By Lemma~\ref{lem:tilde_rln}, the elements $\tilde{t}_{ij}(a)$'s always satisfy (\ref{STP04a}) and (\ref{STP04b}), but do not necessarily satisfy (\ref{STP04c}). A counterexample will appear as a central extension of $\mathfrak{p}_3(R,{}^-)$ in Section~\ref{sec:p3}. There are also central extensions of $\mathfrak{p}_4(R,{}^-)$ which do not satisfy (\ref{STP10}) (see Section~\ref{sec:p4}).
\end{remark}

\begin{proposition}
\label{thm:pm_uce}
 Let $m\geqslant5$ and $(R,{}^-)$ be a unital associative superalgebra with superinvolution. Then the central extension $\psi:\mathfrak{stp}_m(R,{}^-)\rightarrow\mathfrak{p}_m(R,{}^-)$ is universal.
\end{proposition}
\begin{proof}
Let $\varphi:\mathfrak{E}\rightarrow\mathfrak{p}_m(R,{}^-)$ be an arbitrary central extension. It suffices to show there exists a unique homomorphism $\varphi':\mathfrak{stp}_m(R,{}^-)\rightarrow\mathfrak{E}$ such that $\varphi\circ\varphi'=\psi$.
\medskip

For the existence of $\varphi'$, we set $\tilde{t}_{ij}(a),\tilde{f}_{ij}(a),\tilde{g}_{ij}(a)\in\mathfrak{E}$ for $a\in R$ and $1\leqslant i\neq j\leqslant m$ as in (\ref{eq:tilde}). They satisfy all relations (\ref{STP00})-(\ref{STP12}) except (\ref{STP04}) and (\ref{STP10}) as shown in Lemma~\ref{lem:tilde_rln}. Under the additional assumption that $m\geqslant5$, we show that these elements also satisfy (\ref{STP04}) and (\ref{STP10}).

For (\ref{STP04}), since (\ref{STP04a}) and (\ref{STP04b}) hold by Lemma~\ref{lem:tilde_rln}, it suffices to show
\begin{equation}
[\tilde{t}_{ij}(a),\tilde{t}_{ik}(b)]=0,\text{ if }i,j,k\text{ are pairwise distinct.}\tag{\ref{STP04c}}
\end{equation}
The assumption that $m\geqslant5$ allows us to choose $1\leqslant l\leqslant m$ such that $l\neq i,j,k$. As in the proof of Lemma~\ref{lem:tilde_rln}, we let $\tilde{h}_{lj}=[\hat{f}_{lj}(1),\hat{g}_{jl}(1)]$, and deduce from $[\tilde{t}_{ij}(a),\tilde{t}_{ik}(b)]\in\ker\varphi$ that
$$0=[\tilde{h}_{lj},[\tilde{t}_{ij}(a),\tilde{t}_{ik}(b)]]
=[[\tilde{h}_{lj},\tilde{t}_{ij}(a)],\tilde{t}_{ik}(b)]
+[\tilde{t}_{ij}(a),[\tilde{h}_{lj},\tilde{t}_{ik}(b)]]
=[\tilde{t}_{ij}(a),\tilde{t}_{ik}(b)],$$
i.e., (\ref{STP04c}) holds.

For (\ref{STP10}), besides (\ref{STP10a}) and (\ref{STP10b}) in Lemma~\ref{lem:tilde_rln}, we need to show
\begin{equation}
[\tilde{g}_{ij}(a),\tilde{g}_{kl}(b)]=0,\text{ if }i,j,k,l\text{ are pairwise distinct}.\tag{STP10c}\label{STP10c}
\end{equation}
Since $m\geqslant5$, we are permitted to choose $k'$ such that $k'\neq i,j,k,l$. Hence,
\begin{align*}
[\tilde{g}_{ij}(a),\tilde{g}_{kl}(b)]
&=[\tilde{g}_{ij}(a),[\hat{g}_{kk'}(b),\hat{t}_{k'l}(1)]]\\
&=[[\tilde{g}_{ij}(a),\hat{g}_{kk'}(b)],\hat{t}_{k'l}(1)]+(-1)^{(1+|a|)(1+|b|)}[\hat{g}_{kk'}(b),[\tilde{g}_{ij}(a),\hat{t}_{k'l}(1)]].
\end{align*}
Note that $[\tilde{g}_{ij}(a),\hat{g}_{kk'}(b)]\in\ker\psi$ and $[\tilde{g}_{ij}(a),\hat{t}_{k'l}(1)]\in\ker\psi$ are central in $\mathfrak{E}$, (\ref{STP10c}) follows.

Now, the elements $\tilde{t}_{ij}(a),\tilde{f}_{ij}(a),\tilde{g}_{ij}(a)\in\mathfrak{E}$ with $a\in R$ and $1\leqslant i\neq j\leqslant m$ satisfy all relations (\ref{STP00})-(\ref{STP12}). Hence, there is a homomorphism of Lie superalgebras
$$\varphi':\mathfrak{stp}_m(R,{}^-)\rightarrow\mathfrak{E},$$
such that
$$\varphi'(\mathbf{t}_{ij}(a))=\tilde{t}_{ij}(a),\quad
\varphi'(\mathbf{f}_{ij}(a))=\tilde{f}_{ij}(a),\quad
\varphi'(\mathbf{g}_{ij}(a))=\tilde{g}_{ij}(a),$$
for $a\in R$ and $1\leqslant i\neq j\leqslant m$, i.e., $\varphi\circ\varphi'=\psi$.
\medskip

It remains to show the uniqueness of $\varphi'$. Let $\tilde{\varphi}':\mathfrak{stp}_m(R,{}^-)\rightarrow\mathfrak{E}$ be another homomorphism of Lie superalgebras such that $\varphi\circ\tilde{\varphi}'=\psi$. Then
$$\tilde{\varphi}'(\mathbf{t}_{ij}(a))\in\varphi^{-1}({t}_{ij}(a)),\quad
\tilde{\varphi}'(\mathbf{f}_{ij}(a))\in\varphi^{-1}({f}_{ij}(a)),\text{ and }
\tilde{\varphi}'(\mathbf{g}_{ij}(a))\in\varphi^{-1}({g}_{ij}(a)),$$
for $a\in R$ and $1\leqslant i\neq j\leqslant m$. Now, $\mathbf{t}_{ij}(a)=[\mathbf{f}_{ik}(a),\mathbf{g}_{kj}(1)]$. Hence,
$$\tilde{\varphi}'(\mathbf{t}_{ij}(a))
=[\tilde{\varphi}'(\mathbf{f}_{ik}(a)),\tilde{\varphi}'(\mathbf{g}_{kj}(1))]
=[\hat{f}_{ij}(a),\hat{g}_{kj}(1)]=\tilde{t}_{ij}(a)=\varphi'(\mathbf{t}_{ij}(a)).$$
Similarly, 
$$\mathbf{f}_{ij}(a)=[\mathbf{t}_{ik}(1),\mathbf{f}_{kj}(a)]\text{ and }\mathbf{g}_{ij}(a)=[\mathbf{g}_{ik}(a),\mathbf{t}_{kj}(1)]$$ 
imply that $\tilde{\varphi}'(\mathbf{f}_{ij}(a))=\tilde{\varphi}(\mathbf{f}_{ij}(a))$ and $\tilde{\varphi}'(\mathbf{g}_{ij}(a))=\tilde{\varphi}(\mathbf{g}_{ij}(a))$, respectively. Therefore, $\varphi'=\tilde{\varphi}'$ since $\mathfrak{stp}_m(R,{}^-)$ is generated by $\mathbf{t}_{ij}(a),\mathbf{f}_{ij}(a),\mathbf{g}_{ij}(a)$ with $a\in R$ and  $1\leqslant i\neq j\leqslant m$. 
\end{proof}

\begin{theorem}
\label{thm:hml_pm}
Let $m\geqslant5$ and $(R,{}^-)$ be a unital associative superalgebra with superinvolution. Then
$$\mathrm{H}_2(\mathfrak{p}_m(R,{}^-))=\fourIdx{}{+}{}{1}{\mathrm{HD}}(R,{}^-\circ\rho),$$
where $\rho$ is the $\Bbbk$-linear map given in (\ref{eq:rho}) and ${}^-\circ\rho$ is the superinvolution on $R$ obtained by composing the superinvolution ${}^-$ with $\rho$.
\end{theorem}

\begin{proof}
By Proposition~\ref{thm:pm_uce}, the canonical homomorphism $\psi:\mathfrak{stp}_m(R,{}^-)\rightarrow\mathfrak{p}_m(R,{}^-)$ is the universal central extension for $m\geqslant5$. Hence, the second homology of $\mathfrak{p}_m(R,{}^-)$ can be identified with $\ker\psi$, which is isomorphic to $\fourIdx{}{+}{}{1}{\mathrm{HD}}(R,{}^-\circ\rho)$ by Proposition~\ref{prop:stp_ct}.
\end{proof}

\begin{remark}
\label{rmk:hml_com}
If $R$ is super-commutative, then $\fourIdx{}{+}{}{1}{\mathrm{HD}}(R,\mathrm{id})=0$. Hence,
$$\mathrm{H}_2(\mathfrak{p}_m(\Bbbk)\otimes_{\Bbbk}R)
\cong\mathrm{H}_2(\mathfrak{p}_m(R,\rho))
\cong\fourIdx{}{+}{}{1}{\mathrm{HD}}(R,\mathrm{id})
=0.$$
This recovers the results about the second homology of $\mathfrak{p}_m(\Bbbk)\otimes_{\Bbbk}R$ given in \cite{IoharaKoga2005} and \cite{MartinezZelmanov2003}.
\end{remark}
\medskip

In the special case where $(R,{}^-)=(S\oplus S^{\mathrm{op}},\mathrm{ex})$ for a unital associative superalgebra $S$, we have

\begin{corollary}
\label{cor:slmm}
Let $m\geqslant5$ and $S$ be an arbitrary unital associative superalgebra. Then
$$\mathrm{H}_2(\mathfrak{sl}_{m|m}(S))\cong\fourIdx{}{+}{}{1}{\mathrm{HD}}(S\oplus S^{\mathrm{op}},\mathrm{ex}\circ\rho)\cong\mathrm{HC}_1(S),$$
where $\mathrm{HC}_1(S)$ is the first $\mathbb{Z}/2\mathbb{Z}$-graded cyclic homology of $S$ as defined in \cite{ChenSun2015}.
\end{corollary}
\begin{proof}
By Example~\ref{ex:iso_p_SS}, the Lie superalgebra $\mathfrak{sl}_{m|m}(S)$ is isomorphic to $\mathfrak{p}_m(S\oplus S^{\mathrm{op}},\mathrm{ex})$. Theorem~\ref{thm:hml_pm} ensures that 
$$\mathrm{H}_2(\mathfrak{sl}_{m|m}(S))
\cong
\mathrm{H}_2(\mathfrak{p}_m(S\oplus S^{\mathrm{op}},\mathrm{ex}))
\cong
\fourIdx{}{+}{}{1}{\mathrm{HD}}(S\oplus S^{\mathrm{op}},\mathrm{ex}\circ\rho).$$

We need to establish an isomorphism between $\fourIdx{}{+}{}{1}{\mathrm{HD}}(S\oplus S^{\mathrm{op}},\mathrm{ex}\circ\rho)$ and the first $\mathbb{Z}/2\mathbb{Z}$-graded cyclic homology $\mathrm{HC}_1(S)$. It is known from \cite{ChenSun2015} that 
$$\mathrm{HC}_1(S):=\left\{\sum\limits_i\langle a_i,b_i\rangle_{\mathsf{c}}
\in\langle S,S\rangle_{\mathsf{c}}\middle|\sum\limits_i[a_i,b_i]=0\right\},$$
where $\langle S,S\rangle_{\mathsf{c}}=(S\otimes S)/I_{\mathsf{c}}$, and $I_{\mathsf{c}}$ be the $\Bbbk$-submodule of $S\otimes_{\Bbbk}S$ generated by 
$$a\otimes b-(-1)^{|a||b|}b\otimes a\text{ and  }(-1)^{|a||c|}ab\otimes c+(-1)^{|b||a|}bc\otimes a+(-1)^{|c||b|}ca\otimes b,\quad a,b,c\in S.$$
Indeed, $\fourIdx{}{+}{}{1}{\mathrm{HD}}(S\oplus S^{op},\mathrm{ex}\circ\rho)$ is a $\Bbbk$-submodule of $\langle S\oplus S^{\mathrm{op}},S\oplus S^{\mathrm{op}}\rangle$. The $\Bbbk$-module $\langle S\oplus S^{\mathrm{op}},S\oplus S^{\mathrm{op}}\rangle$ can be identified with $\langle S,S\rangle_{\mathsf{c}}$ as follows: By (\ref{eq:HD1}), we deduce that
\begin{align*}
\langle a\oplus0,0\oplus b\rangle
=&\langle (a\oplus0)(1\oplus0),0\oplus b\rangle+0+0\\
=&\langle (a\oplus0)(1\oplus0),0\oplus b\rangle
+(-1)^{|a||b|}\langle (1\oplus0)(0\oplus b),a\oplus0\rangle\\
&+(-1)^{|a||b|}\langle (0\oplus b)(a\oplus 0),1\oplus0\rangle\\
=&0,
\end{align*}
and then $\langle 0\oplus a,b\oplus0\rangle=0$. It follows that
$$\langle a_1\oplus a_2,b_1\oplus b_2\rangle=\langle a_1\oplus0,b_1\oplus0\rangle
+\langle a_2\oplus0,b_2\oplus0\rangle,\quad a_1,a_2,b_1,b_2\in S.$$
Consequently, 
$$\langle a_1\oplus a_2,b_1\oplus b_2\rangle\mapsto\langle a_1,b_1\rangle_{\mathsf{c}}
+\langle a_2,b_2\rangle_{\mathsf{c}}$$
defines an isomorphism $\langle S\oplus S^{\mathrm{op}},S\oplus S^{\mathrm{op}}\rangle\rightarrow\langle S,S\rangle_{\mathsf{c}}$. Its restriction to $\fourIdx{}{+}{}{1}{\mathrm{HD}}(S\oplus S^{\mathrm{op}},\mathrm{ex}\circ\rho)$ gives an isomorphism onto 
$\mathrm{HC}_1(S)$.
\end{proof}

\begin{remark}
This corollary recovers the second homology of $\mathfrak{sl}_{m|m}(S)$ for $m\geqslant5$ obtained in \cite{ChenSun2015}. As a byproduct, we obtain the isomorphism
$$\fourIdx{}{+}{}{1}{\mathrm{HD}}(S\oplus S^{\mathrm{op}},\mathrm{ex}\circ\rho)\cong\mathrm{HC}_1(S),$$
which indicates that the first $\mathbb{Z}/2\mathbb{Z}$-graded cyclic homology can be realized via the first $\mathbb{Z}/2\mathbb{Z}$-graded dihedral homology. However, it is unknown yet whether an isomorphism between a higher degree cyclic homology and a higher degree dihedral homology exists.
\end{remark}

\section{Exceptional Case I: Second Homology of $\mathfrak{p}_4(R,{}^-)$}
\label{sec:p4}

Different from the universal central extension of $\mathfrak{p}_m(R,{}^-)$ for $m\geqslant5$ discussed in the previous sections, the central extension $\psi:\mathfrak{stp}_4(R,{}^-)\rightarrow\mathfrak{p}_4(R,{}^-)$ is not necessarily universal. As we will exhibit in this section, the universal central extension of $\mathfrak{stp}_4(R,{}^-)$ can be obtained via creating a $2$-cocycle with values in the $\Bbbk$-module $R/(R_{(-)}\cdot R)$, where $R_{(-)}\cdot R$ is the right ideal of $R$ generated by $\bar{a}-\rho(a)$ for $a\in R$. 
\medskip

Our primary aim is to define a $2$-cocycle on $\mathfrak{stp}_4(R,{}^-)$ with values in $R/(R_{(-)}\cdot R)$. Note that $\mathfrak{stp}_4(R,{}^-)$ is decomposed as a direct sum of Lie sub-superalgebras:
$$\mathfrak{stp}_4(R,{}^-)=\mathfrak{a}\oplus\mathfrak{b}$$
where
\begin{align*}
\mathfrak{a}:&=\mathrm{span}_{\Bbbk}\{\mathbf{h}_{ij}(a,b), \mathbf{t}_{ij}(a),\mathbf{f}_i(a),\mathbf{f}_{ij}(a)|a,b\in R, 1\leqslant i\neq j\leqslant 4\},\\
\mathfrak{b}:&=\mathrm{span}_{\Bbbk}\{\mathbf{g}_i(a),\mathbf{g}_{ij}(a)|a\in R, 1\leqslant i\neq j\leqslant 4\}.
\end{align*}
We first define a $\Bbbk$-linear map $\beta_0:\mathfrak{b}\times\mathfrak{b}\rightarrow R/(R_{(-)}\cdot R)$ by
\begin{eqnarray*}
&\beta_0(\mathbf{g}_{ij}(a),\mathbf{g}_{kl}(b))=\epsilon(ijkl)\boldsymbol{\pi}(a\cdot\rho(b)),&\\
&\beta_0(\mathbf{g}_{i}(a),\mathfrak{b})=\beta_0(\mathfrak{b},\mathbf{g}_i(a))=0,&
\end{eqnarray*}
for $a,b\in R$, $1\leqslant i\neq j\leqslant 4$ and $1\leqslant k\neq l\leqslant 4$, where 
$$\boldsymbol{\pi}:R\rightarrow R/(R_{(-)}\cdot R)$$
is the canonical quotient map of $\Bbbk$-modules, and 
$$\epsilon(ijkl)=\begin{cases}\text{the sign of the permutation }(ijkl),&\text{if }\{i,j,k,l\}=\{1,2,3,4\},\\0,&\text{otherwise.}\end{cases}$$
The $\Bbbk$-linear map $\beta_0$ is well-defined since $\mathbf{g}_{ij}(\bar{a})=-\mathbf{g}_{ji}(\rho(a))$ and $\boldsymbol{\pi}$ satisfies
\begin{equation}
\boldsymbol{\pi}(\bar{a}b)=\boldsymbol{\pi}(\rho(a)b), \quad a,b\in R.\label{eq:p4pi}
\end{equation} 

The $\Bbbk$--bilinear map $\beta_0:\mathfrak{b}\times\mathfrak{b}\rightarrow R/(R_{(-)}\cdot R)$ is extended to a $\Bbbk$--bilinear map
$$\beta:\mathfrak{stp}_4(R,{}^-)\times\mathfrak{stp}_4(R,{}^-)\rightarrow R/(R_{(-)}\cdot R),$$
such that $\mathfrak{a}$ lies in the radical of $\beta$, i.e.,
$$\beta(\mathfrak{a},\mathfrak{stp}_4(R,{}^-))=
\beta(\mathfrak{stp}_4(R,{}^-),\mathfrak{a})=0.$$
The fact that $\beta$ is a $2$-cocycle on $\mathfrak{stp}_4(R,{}^-)$ is verified in the following lemma.

\begin{lemma}
\label{lem:stp4_2cy2}
The $\Bbbk$-bilinear map $\beta$ is a 2-cocycle on $\mathfrak{stp}_4(R,{}^-)$ with values in $R/(R_{(-)}\cdot R)$.
\end{lemma}
\begin{proof}
We have to show $\beta$ satisfies
\begin{eqnarray}
&\beta(y,x)=-(-1)^{|x||y|}\beta(x,y),&\label{stp4:eq:2cyl1}\\
&(-1)^{|x||z|}\beta([x,y],z)+(-1)^{|y||x|}\beta([y,z],x)+(-1)^{|z||y|}\beta([z,x],y)=0,&\label{stp4:eq:2cyl2}
\end{eqnarray}
for $x,y,z\in\mathfrak{stp}_4(R,{}^-)$.
\medskip

For (\ref{stp4:eq:2cyl1}), it suffices to show
\begin{equation}
\beta(\mathbf{g}_{ij}(a),\mathbf{g}_{kl}(b))=-(-1)^{(1+|a|)(1+|b|)}\beta(\mathbf{g}_{kl}(b),\mathbf{g}_{ij}(a)),\label{stp4:eq:2cyl11}
\end{equation}
for $a,b\in R$, $i\neq j$ and $k\neq l$. By (\ref{eq:p4pi}), we have
$$\boldsymbol{\pi}(ab)=
\boldsymbol{\pi}(\overline{ab})=
(-1)^{|a||b|}\boldsymbol{\pi}(\bar{b}\bar{a})
=(-1)^{|a||b|}\boldsymbol{\pi}(b\bar{a})=\boldsymbol{\pi}(a\bar{b})=\boldsymbol{\pi}(\bar{a}\bar{b})=(-1)^{|a||b|}\boldsymbol{\pi}(ba),$$
which implies (\ref{stp4:eq:2cyl11}).
\medskip

For (\ref{stp4:eq:2cyl2}), since $\mathfrak{a}$ is a Lie sub-superalgebra of $\mathfrak{stp}_4(R,{}^-)$  included in the radical of $\beta$, we know that all the terms of (\ref{stp4:eq:2cyl2}) vanish if at least two of $x$, $y$ and $z$ lie in $\mathfrak{a}$. Moreover, $[\mathfrak{b},\mathfrak{b}]=0$ implies that (\ref{stp4:eq:2cyl2}) is trivial if $x$, $y$ and $z$ are contained in $\mathfrak{b}$. Note also that (\ref{stp4:eq:2cyl2}) is symmetric with respect to all permutations on $\{x,y,z\}$, the proof reduces to the special case where $x\in\mathfrak{a}$ and $y,z\in\mathfrak{b}$.
In this situation, $[y,z]=0$, then (\ref{stp4:eq:2cyl2}) is simplified to
\begin{equation}
\beta([y,x],z)=(-1)^{|x||y|+|y||z|+|z||x|}\beta([z,x],y).\label{eq:stp4_2cy4}
\end{equation}

If $x=\mathbf{f}_{ij}(a)$ or $x=\mathbf{f}_{i}(a)$, then $[x,\mathfrak{b}]\subseteq\mathfrak{a}$ is included in the radical of $\beta$, and hence both sides of (\ref{eq:stp4_2cy4}) are zero. (\ref{eq:stp4_2cy4}) is also trivial if $y=\mathbf{g}_i(a)$ and $z=\mathbf{g}_j(a)$. Now, it remains to verify (\ref{eq:stp4_2cy4}) in the following four cases:
\begin{enumerate}
\item $x=\mathbf{t}_{rs}(c)$, $y=\mathbf{g}_{ij}(a)$ and $z=\mathbf{g}_{kl}(b)$,
\item $x=\mathbf{t}_{rs}(c)$, $y=\mathbf{g}_i(a)$ and $z=\mathbf{g}_{kl}(b)$,
\item $x=\mathbf{h}_{rs}(c,c')$, $y=\mathbf{g}_{ij}(a)$ and $z=\mathbf{g}_{kl}(b)$,
\item $x=\mathbf{h}_{rs}(c,c')$, $y=\mathbf{g}_{i}(a)$ and $z=\mathbf{g}_{kl}(b)$,
\end{enumerate}
where $a,b,c,c'\in R$ and $1\leqslant i,j,k,l,r,s\leqslant4$.

For Case (i), the definition of $\beta$ and (\ref{STP07}) imply that
\begin{align*}
\beta([y,x],z)&=(-1)^{|b|}(\delta_{jr}\epsilon(iskl)-\delta_{ir}\epsilon(jskl))\boldsymbol{\pi}(acb),\\
\beta([z,x],y)&=(-1)^{|a|}(\delta_{lr}\epsilon(ksij)-\delta_{kr}\epsilon(lsij))\boldsymbol{\pi}(bca).
\end{align*}
Since $\boldsymbol{\pi}(abc)=(-1)^{|a||b|+|b||c|+|c||a|}\boldsymbol{\pi}(cba)$ and
$$\delta_{jr}\epsilon(iskl)-\delta_{ir}\epsilon(jskl)
=\delta_{lr}\epsilon(ksij)-\delta_{kr}\epsilon(lsij),$$
we obtain (\ref{eq:stp4_2cy4}) in Case (i). Similar arguments show that (\ref{eq:stp4_2cy4}) also holds in Cases (ii)-(iv).
\end{proof}
\medskip

The $2$-cocycle $\beta:\mathfrak{stp}_4(R,{}^-)\times\mathfrak{stp}_4(R,{}^-)
\rightarrow R/(R_{(-)}\cdot R)$ gives rise to a new Lie superalgebra
$$\widehat{\mathfrak{stp}}_4(R,{}^-):=\mathfrak{stp}_4(R,{}^-)\oplus (R/(R_{(-)}\cdot R)),$$
under the super-bracket
$$[x\oplus c,y\oplus c]:=[x,y]\oplus \beta(x,y),$$
for $x,y\in\mathfrak{stp}_4(R,{}^-)$ and $c,c'\in R/(R_{(-)}\cdot R)$. The canonical projection $\psi'_4:\widehat{\mathfrak{stp}}_4(R,{}^-)\rightarrow\mathfrak{stp}_4(R,{}^-)$ is a central extension. Furthermore, we have

\begin{proposition}
\label{prop:p4_uce}
Let $(R,{}^-)$ be a unital associative superalgebra with superinvolution. Then the central extension $\psi':\widehat{\mathfrak{stp}}_4(R,{}^-)\rightarrow \mathfrak{stp}_4(R,{}^-)$ is universal.
\end{proposition}
\begin{proof}
Since $\psi:\mathfrak{stp}_4(R,{}^-)\rightarrow\mathfrak{p}_4(R,{}^-)$ is a central extension, it suffices to show that the central extension $\psi\circ\psi':\widehat{\mathfrak{stp}}_4(R,{}^-)\rightarrow\mathfrak{p}_4(R,{}^-)$ is universal. Let $\varphi:\mathfrak{E}\rightarrow\mathfrak{p}_4(R,{}^-)$ be an arbitrary central extension of $\mathfrak{p}_4(R,{}^-)$. We have to show that there exists a unique homomorphism $\varphi':\widehat{\mathfrak{stp}}_4(R,{}^-)\rightarrow\mathfrak{E}$ such that $\varphi\circ\varphi'=\psi\circ\psi'$.

We first prove the existence of $\varphi'$. Take $\tilde{t}_{ij}(a)$, $\tilde{f}_{ij}(a)$ and $\tilde{g}_{ij}(a)\in\mathfrak{E}$ as in (\ref{eq:tilde}). By Lemma~\ref{lem:tilde_rln} and the same arguments as in Theorem~\ref{thm:pm_uce}, these elements satisfy (\ref{STP00})-(\ref{STP12}) except (\ref{STP10}). Instead of (\ref{STP10}), we will show that
\begin{equation}
[\tilde{g}_{ij}(a),\tilde{g}_{kl}(b)]=\epsilon(ijkl)\tilde{\pi}(a\cdot\rho(b)),\text{ for }i\neq j\text{ and }k\neq l,\label{eq:p4rln10}
\end{equation}
where
$$\tilde{\pi}(a):=[\tilde{g}_{12}(a),\tilde{g}_{34}(1)]$$
is central in $\mathfrak{E}$ since $\varphi(\tilde{\pi}(a))=0$. 

If $(i,j,k,l)=(1,3,4,2)$, we have $\tilde{g}_{42}(b)=[\tilde{g}_{43}(1),\tilde{t}_{32}(b)]\pmod{\ker\varphi}
$. Then the Jacobi identity implies
$$[\tilde{g}_{13}(a),\tilde{g}_{42}(b)]
=[[\tilde{g}_{13}(a),\tilde{g}_{43}(1)],\tilde{t}_{32}(b)]
+(-1)^{|a|}[\tilde{g}_{43}(1),[\tilde{g}_{13}(a),\tilde{t}_{32}(b)]].$$
Since $[\tilde{g}_{13}(a),\tilde{g}_{43}(1)]\in\ker\varphi$ and $\tilde{g}_{12}(ab)=[\tilde{g}_{13}(1),\tilde{t}_{32}(b)]\pmod{\ker\varphi}$, we conclude that
$$[\tilde{g}_{13}(a),\tilde{g}_{42}(b)]
=(-1)^{|a|}[\tilde{g}_{43}(1),\tilde{g}_{12}(ab)]=(-1)^{|b|}[\tilde{g}_{12}(ab),\tilde{g}_{34}(1)]=\tilde{\pi}(a\cdot\rho(b)).$$
For other choices of $(i,j,k,l)$ such that $i,j,k,l$ are pairwise distinct, (\ref{eq:p4rln10}) is also checked similarly. Combining with (\ref{STP10a}) and (\ref{STP10b}), we obtain (\ref{eq:p4rln10}).

Moreover, we deduce that $\tilde{\pi}(R_{(-)}\cdot R)=0$ by  the following computation
\begin{align*}
\tilde{\pi}(ab)&=[\tilde{g}_{12}(ab),\tilde{g}_{34}(1)]
=[[\tilde{g}_{13}(a),\tilde{t}_{32}(b)],\tilde{g}_{34}(1)]\\
&=[\tilde{g}_{13}(a),[\tilde{t}_{32}(b),\tilde{g}_{34}(1)]]
=-[\tilde{g}_{13}(a),\tilde{g}_{24}(\bar{b})]\\
&=-[[\tilde{g}_{14}(1),\tilde{t}_{43}(a)],\tilde{g}_{24}(\bar{b})]
=-[\tilde{g}_{14}(1),[\tilde{t}_{43}(a),\tilde{g}_{24}(\bar{b})]]\\
&=(-1)^{|a|(1+|b|)}[\tilde{g}_{14}(1),\tilde{g}_{23}(\bar{b}a)]
=(-1)^{|a|(1+|b|)}[\tilde{g}_{14}(1),[\tilde{g}_{21}(\bar{b}a),\tilde{t}_{13}(1)]]\\
&=(-1)^{|a|(1+|b|)}(-1)^{1+|a|+|b|}[\tilde{g}_{21}(\bar{b}a),[\tilde{g}_{14}(1),\tilde{t}_{13}(1)]]
=-(-1)^{|b|+|a||b|}[\tilde{g}_{21}(\bar{b}a),\tilde{g}_{34}(1)]\\
&=(-1)^{|a|}[\tilde{g}_{12}(\bar{a}b),\tilde{g}_{34}(1)]
=(-1)^{|a|}\tilde{\pi}(\bar{a}b).
\end{align*}

Therefore, there is a homomorphism
$\varphi':\widehat{\mathfrak{stp}}_4(R,{}^-)\rightarrow\mathfrak{E}$ such that
\begin{align*}
\varphi'(\mathbf{t}_{ij}(a)\oplus0)&=\tilde{t}_{ij}(a),&
\varphi'(\mathbf{f}_{ij}(a)\oplus0)&=\tilde{f}_{ij}(a),\\
\varphi'(\mathbf{g}_{ij}(a)\oplus0)&=\tilde{g}_{ij}(a),&
\varphi'(0\oplus \boldsymbol{\pi}(a))&=\tilde{\pi}(a),
\end{align*}
for $a\in R$ and $1\leqslant i\neq j\leqslant4$, i.e., $\psi\circ\psi'=\varphi\circ\varphi'$. The uniqueness of $\varphi'$ follows from a similar argument as in the proof of Proposition~\ref{thm:pm_uce}. 
\end{proof}

Since the second homology of $\mathfrak{p}_m(R,{}^-)$ is isomorphic to the kernel of its universal central extension, we conclude that

\begin{theorem}
\label{thm:p4_hml}
Let $(R,{}^-)$ be a unital associative superalgebra with superinvolution. Then
$$\mathrm{H}_2(\mathfrak{p}_4(R,{}^-))=\fourIdx{}{+}{}{1}{\mathrm{HD}}(R,{}^-\circ\rho)\oplus R/(R_{(-)}\cdot R).\qed$$
\end{theorem}

\begin{remark}
\label{rmk:p4_hml_com}
If $R$ is super-commutative, then $R_{(-)}=0$. In this situation,
$$\mathrm{H}_2(\mathfrak{p}_4(\Bbbk)\otimes_{\Bbbk}R)\cong
\mathrm{H}_2(\mathfrak{p}_4(R,\rho))
\cong \fourIdx{}{+}{}{1}{\mathrm{HD}}(R,\mathrm{id})\oplus R\cong
R,$$
which coincides with the second homology of $\mathfrak{p}_4(\Bbbk)\otimes_{\Bbbk}R$ obtained in \cite{IoharaKoga2005}.
\end{remark}
\medskip

In the special case where $(R,{}^-)=(S\oplus S^{\mathrm{op}},\mathrm{ex})$, Theorem~\ref{thm:p4_hml} recovers the second homology of $\mathfrak{sl}_{4|4}(S)$ given in \cite{ChenSun2015}.

\begin{corollary}
\label{cor:sl44}
Let $S$ be an arbitrary unital associative superalgebra. Then
$$\mathrm{H}_2(\mathfrak{sl}_{4|4}(S))=\mathrm{HC}_1(S).$$
\end{corollary}
\begin{proof}
It is known from Example~\ref{ex:iso_p_SS} that $\mathfrak{sl}_{4|4}(S)$ is isomorphic to $\mathfrak{p}_4(S\oplus S^{\mathrm{op}},\mathrm{ex})$. Hence,
$$\mathrm{H}_2(\mathfrak{sl}_{4|4}(S))
\cong\mathrm{H}_2(\mathfrak{p}_4(S\oplus S^{\mathrm{op}},\mathrm{ex}))
\cong\fourIdx{}{+}{}{1}{\mathrm{HD}}(S\oplus S^{\mathrm{op}},\mathrm{ex}\circ\rho)\oplus R/(R_{(-)}\cdot R),$$
where $(R,{}^-)=(S\oplus S^{\mathrm{op}},\mathrm{ex})$. Moreover, $R/(R_{(-)}\cdot R)=0$ since $R_{(-)}$ contains a unit element $1\oplus(-1)$. It follows that $\mathrm{H}_2(\mathfrak{sl}_{4|4}(S))$ is isomorphic to 
$\fourIdx{}{+}{}{1}{\mathrm{HD}}(S\oplus S^{\mathrm{op}},\mathrm{ex}\circ\rho)$, which is further identified with
$\mathrm{HC}_1(S)$ by Corollary~\ref{cor:slmm}.
\end{proof}

\section{Exceptional Case II: Second Homology of $\mathfrak{p}_3(R,{}^-)$}
\label{sec:p3}

Similar to the exceptional case of $\mathfrak{p}_4(R,{}^-)$, the central extension $\psi:\mathfrak{stp}_3(R,{}^-)\rightarrow\mathfrak{p}_3(R,{}^-)$ is not necessarily universal. The key ingredient of this section is an explicit construction of the universal central extension of $\mathfrak{stp}_3(R,{}^-)$, which involves a $2$-cocycle taking values in the $\Bbbk$-module
$$\mathfrak{z}:=\frac{R}{3R+R_{(-)}\cdot R}\oplus\frac{R}{3R+R_{(-)}\cdot R}\oplus\frac{R}{3R+R_{(-)}\cdot R},$$
where $R_{(-)}\cdot R$ is the right ideal of $R$ generated by $\bar{a}-\rho(a)$ for $a\in R$.

According to Lemma~\ref{prop:stp_tridec}, $\mathfrak{stp}_3(R,{}^-)$ is spanned as a $\Bbbk$--module by
$$\mathfrak{B}:=\{\mathbf{h}_{ij}(a,b),\mathbf{t}_{ij}(a),\mathbf{f}_{ij}(a),\mathbf{g}_{ij}(a),
\mathbf{f}_k(a),
\mathbf{g}_k(a)|a,b\in R, 1\leqslant i,j,k\leqslant 3\text{ with }i\neq j\}.$$
We define a $\Bbbk$--bilinear map $\beta:\mathfrak{stp}_3(R,{}^-)\times\mathfrak{stp}_3(R,{}^-)\rightarrow\mathfrak{z}$ as follows:
\begin{align*}
\beta(\mathbf{t}_{ij}(a),\mathbf{t}_{ik}(b))
&=\epsilon(ijk)\boldsymbol{\pi}_i(ab),&\text{ for }a,b\in R,\text{ and pairwise distinct }i,j,k,\\
\beta(\mathbf{f}_i(a),\mathbf{g}_{jk}(b))
&=\epsilon(ijk)\boldsymbol{\pi}_i(ab),&\text{ for }a,b\in R,\text{ and pairwise distinct }i,j,k,\\
\beta(\mathbf{g}_{jk}(b),\mathbf{f}_i(a))
&=-(-1)^{(1+|a|)(1+|b|)}\epsilon(ijk)\boldsymbol{\pi}_i(ab),&\text{ for }a,b\in R,\text{ and pairwise distinct }i,j,k,\\
\beta(x,y)&=0,&\text{for other pairs }(x,y)\in\mathfrak{B}\times\mathfrak{B},
\end{align*}
where $\epsilon(ijk)$ is the sign of the permutation $(ijk)$ and $\boldsymbol{\pi}_i:R\rightarrow R/(3R+R_{(-)}\cdot R)$ denotes the canonical quotient of $R$ onto the $i$-th direct summand of $\mathfrak{z}$ for $i=1,2,3$.

\begin{lemma}
\label{lem:stp3_2cyl}
The well-defined $\Bbbk$--bilinear map $\beta:\mathfrak{stp}_m(R,{}^-)\times\mathfrak{stp}_m(R,{}^-)\rightarrow\mathfrak{z}$ is a $2$-cocycle.
\end{lemma}
\begin{proof}
The well-definedness of $\beta$ is clear since
$$\mathbf{f}_i(\bar{a})=\mathbf{f}_i(\rho(a)),\quad
\mathbf{f}_{ij}(\bar{a})=\mathbf{f}_{ji}(\rho(a)),\quad
\mathbf{g}_{i}(\bar{a})=-\mathbf{g}_i(\rho(a)),\quad
\mathbf{g}_{ij}(\bar{a})=-\mathbf{g}_{ji}(\rho(a)),$$ 
for $a\in R$ and $1\leqslant i\neq j\leqslant 3$ and $\boldsymbol{\pi}_i(\bar{a}b)=\boldsymbol{\pi}_i(\rho(a)b)$ for $a,b\in R$ and $i=1,2,3$.

As in the proof of Lemma~\ref{lem:stp4_2cy2}, we deduce that $\boldsymbol{\pi}_i(ab)=(-1)^{|a||b|}\boldsymbol{\pi}_i(ba)$ for $i=1,2,3$, which yields that the $\Bbbk$-bilinear map $\beta$ satisfies
\begin{equation*}
\beta(x,y)=-(-1)^{|x||y|}\beta(y,x),\quad x,y\in\mathfrak{stp}_3(R,{}^-).
\end{equation*}
It remains to show
\begin{equation}
J(x,y,z):=(-1)^{|x||z|}\beta([x,y],z)
+(-1)^{|y||x|}\beta([y,z],x)
+(-1)^{|z||y|}\beta([z,x],y)
=0,\label{eq:stp3_2cyl2}
\end{equation}
for $x,y,z\in\mathfrak{stp}_3(R,{}^-)$.  Since (\ref{eq:stp3_2cyl2}) is symmetric under all permutations of $\{x,y,z\}$, we may assume $\beta([x,y],z)\neq0$, which only occurs when $z=\mathbf{t}_{ik}(a)$, $z=\mathbf{g}_{jk}(a)$, or $z=\mathbf{f}_i(a)$.

If $z=\mathbf{t}_{ik}(a)$ for $a\in R$ and $1\leqslant i\neq k\leqslant 3$, we have to verify  that $J(x,y,\mathbf{t}_{ik}(a))=0$ for all pairs $(x,y)\in\mathfrak{B}\times\mathfrak{B}$ such that $\beta([x,y],\mathbf{t}_{ik}(a))\neq0$, which occurs only when $(x,y)\in\mathfrak{B}\times\mathfrak{B}$ is one of the following pairs:
$$(\mathbf{h}_{ij}(a,a'), \mathbf{t}_{ij}(b)), 
(\mathbf{h}_{ik}(a,a'), \mathbf{t}_{ij}(b)),
(\mathbf{t}_{ik}(b), \mathbf{t}_{kj}(c)),
(\mathbf{f}_{ik}(b), \mathbf{g}_{kj}(c)),
(\mathbf{f}_i(b), \mathbf{g}_{ij}(c)),
(\mathbf{f}_{ij}(b), \mathbf{g}_{j}(c)),$$
where $a,a', b,c\in R$ and $j$ is the unique element in $\{1,2,3\}\backslash\{i,k\}$. The verification is straightforward. We omit the tedious details here. 

Similarly, $J(x,y,z)=0$ when $z=\mathbf{g}_{jk}(a)$ or $z=\mathbf{f}_i(a)$.
\end{proof}
\bigskip

The $2$-cocycle $\beta:\mathfrak{stp}_3(R,{}^-)\times\mathfrak{stp}_3(R,{}^-)\rightarrow\mathfrak{z}$ determines a central extension
$$\psi'_3:\mathfrak{stp}_3(R,{}^-)\oplus\mathfrak{z}\rightarrow\mathfrak{stp}_3(R,{}^-),$$
where $\psi'_3$ is the canonical projection and the super-bracket on $\mathfrak{stp}_3(R,{}^-)\oplus\mathfrak{z}$ is given by
$$[x\oplus c,y\oplus c']=[x,y]\oplus \beta(x,y),\quad x,y\in\mathfrak{stp}_3(R,{}^-),\text{ and }c,c'\in\mathfrak{z}.$$
Moreover, we have the following proposition.

\begin{proposition}
\label{prop:p3_uce}
The central extension $\psi'_3:\mathfrak{stp}_3(R,{}^-)\oplus\mathfrak{z}\rightarrow\mathfrak{stp}_3(R,{}^-)$ is universal.
\end{proposition}

\begin{proof}
It suffices to show that the central extension $\psi\circ\psi'_3:\mathfrak{stp}_3(R,{}^-)\oplus\mathfrak{z}\rightarrow\mathfrak{p}_3(R,{}^-)$ is universal.

Let $\varphi:\mathfrak{E}\rightarrow\mathfrak{p}_3(R,{}^-)$ be an arbitrary central extension of $\mathfrak{p}_3(R,{}^-)$. Pick elements $\tilde{t}_{ij}(a)$, $\tilde{f}_{ij}(a)$, $\tilde{g}_{ij}(a)\in\mathfrak{E}$ with $a\in R$ and  $1\leqslant i\neq j\leqslant3$ as in (\ref{eq:tilde}). By Lemma~\ref{lem:tilde_rln}, they satisfy all relations (\ref{STP00})-(\ref{STP12}) except (\ref{STP04}) and (\ref{STP10}). Moreover, there are no four pairwise distinct indices $1\leqslant i,j,k,l\leqslant3$. Hence, (\ref{STP10a}) and (\ref{STP10b}) imply (\ref{STP10}).

For $i\in\{1,2,3\}$, there exist unique $j$ and $k$ such that $(ijk)$ is an even permutation of $\{1,2,3\}$. Let
\begin{equation*}
\tilde{\pi}_i(a):=[\tilde{t}_{ij}(1),\tilde{t}_{ik}(a)]\in\ker\varphi,
\end{equation*}
for $a\in R$ and $i=1,2,3$. We claim that $\tilde{\pi}_i(3R+R_{(-)}\cdot R)=0$ for $i=1,2,3$. 

Firstly, we show $\tilde{\pi}_i(3R)=0$. Take $\tilde{h}_{ij}=[\hat{f}_{ij}(1),\hat{g}_{ji}(1)]$, then 
$$[\tilde{h}_{ij},\tilde{t}_{ij}(1)]=2\tilde{t}_{ij}(1)\pmod{\ker\varphi},\text{ and }[\tilde{h}_{ij},\tilde{t}_{ik}(a)]=\tilde{t}_{ik}(a)\pmod{\ker\varphi}.$$
Hence, $\tilde{\pi}(a)\in\ker\varphi$ implies that
\begin{align*}
0=[\tilde{h}_{ij},\tilde{\pi}_i(a)]
=[\tilde{h}_{ij},[\tilde{t}_{ij}(1),\tilde{t}_{ik}(a)]]
=3[\tilde{t}_{ij}(1),\tilde{t}_{ik}(a)]
=\tilde{\pi}_i(3a),
\end{align*}
i.e.,  $\tilde{\pi}_i(3R)=0$.

Secondly, we prove that
$\tilde{\pi}_i(ab)=(-1)^{|a|}\tilde{\pi}_i(\bar{a}b)$
for $a,b\in R$ and $i=1,2,3$. Observing that
$$[\tilde{f}_{ij}(a),\tilde{g}_{jk}(b)]=\tilde{t}_{ik}(ab)\pmod{\ker\varphi},\text{ and }[\tilde{t}_{ij}(1),\tilde{g}_{jk}(b)]=0\pmod{\ker\varphi},$$
we have
$$\tilde{\pi}_i(ab)=[\tilde{t}_{ij}(1),\tilde{t}_{ik}(ab)]
=[\tilde{t}_{ij}(1),[\tilde{f}_{ij}(a),\tilde{g}_{jk}(b)]]
=(-1)^{|a|}[[\tilde{t}_{ij}(1),\tilde{f}_{ji}(\bar{a})],\tilde{g}_{jk}(b)].$$
Moreover, $[\tilde{t}_{ij}(1),\tilde{f}_{ji}(\bar{a})]=[\tilde{t}_{ik}(1),\tilde{f}_{kj}(\bar{a})]\pmod{\ker\varphi}$. It follows that
$$\tilde{\pi}_i(ab)=(-1)^{|a|}[[\tilde{t}_{ik}(1),\tilde{f}_{ki}(\bar{a})],\tilde{g}_{jk}(b)]
=-(-1)^{|b|}[\tilde{t}_{ik}(1),\tilde{t}_{ij}(a\bar{b})]$$
Now, $\tilde{t}_{ik}(1)=[\tilde{t}_{ij}(1),\tilde{jk}(1)]\pmod{\ker\varphi}$. We further deduce that
$$\tilde{\pi}_i(ab)=-(-1)^{|b|}[[\tilde{t}_{ij}(1),\tilde{t}_{jk}(1)],\tilde{t}_{ij}(a\bar{b})]
=(-1)^{|b|}[\tilde{t}_{ij}(1),\tilde{t}_{ik}(a\bar{b})]
=(-1)^{|b|}\tilde{\pi}_i(a\bar{b}),$$
which yields
\begin{align*}
\tilde{\pi}_i(ab)
&=(-1)^{|a|+|b|}\tilde{\pi}_i(\overline{ab})
=(-1)^{|a|+|b|+|a||b|}\tilde{\pi}_i(\bar{b}\bar{a})
=(-1)^{|b|+|a||b|}\tilde{\pi}_i(\bar{b}a)\\
&=(-1)^{|a|+|a||b|}\tilde{\pi}_i(\overline{\bar{b}a})
=(-1)^{|a|}\tilde{\pi}_i(\bar{a}b).
\end{align*}

Finally, we show that 
\begin{equation}
[\tilde{t}_{ij}(a),\tilde{t}_{ik}(b)]=\epsilon(ijk)\tilde{\pi}_i(ab),\text{ for }\{i,j,k\}=\{1,2,3\}.\label{eq:p3stp10}
\end{equation}
If the permutation $(ijk)$ is even, then we deduce from $[\tilde{t}_{ij}(1),\tilde{t}_{ij}(a)]=0\pmod{\ker\varphi}$ that
$$[\tilde{t}_{ij}(a),\tilde{t}_{ik}(b)]
=[\tilde{t}_{ij}(a),[\tilde{t}_{ij}(1),\tilde{t}_{jk}(b)]]
=[\tilde{t}_{ij}(1),[\tilde{t}_{ij}(a),\tilde{t}_{jk}(b)]]
=[\tilde{t}_{ij}(1),\tilde{t}_{ik}(ab)]
=\tilde{\pi}_i(ab).$$
If the permutation $(ijk)$ has negative sign, then $(ikj)$ has positive sign and
\begin{align*}
[\tilde{t}_{ij}(a),\tilde{t}_{ik}(b)]
=-(-1)^{|a||b|}[\tilde{t}_{ik}(b),\tilde{t}_{ij}(a)]
=-(-1)^{|a||b|}\tilde{\pi}_i(ba)
=-\tilde{\pi}_i(ab).
\end{align*}
We thus obtain (\ref{eq:p3stp10}).

Therefore, there is a homomorphism of Lie superalgebras $\varphi':\widehat{\mathfrak{stp}}_3(R,{}^-)\rightarrow\mathfrak{E}$ such that
$$\varphi'(\mathbf{t}_{ij}(a))=\tilde{t}_{ij}(a),
\quad\varphi'(\mathbf{f}_{ij}(a))=\tilde{f}_{ij}(a),
\quad\varphi'(\mathbf{g}_{ij}(a))=\tilde{g}_{ij}(a),$$
for $a\in R$ and $1\leqslant i\neq j\leqslant 3$, i.e., $\varphi\circ\varphi'=\psi\circ\psi'$. The homomorphism $\varphi'$ is unique as in the proof of Proposition~\ref{thm:pm_uce}.
\end{proof}

Now, we conclude from Propositions~\ref{prop:stp_ce} and~\ref{prop:p3_uce} that:

\begin{theorem}
\label{thm:p3_hml}
Let $(R,{}^-)$ be a unital associative superalgebra with superinvolution. Then
$$\mathrm{H}_2(\mathfrak{p}_3(R,{}^-))=\fourIdx{}{+}{}{1}{\mathrm{HD}}(R,{}^-\circ\rho)\oplus\frac{R}{3R+R_{(-)}\cdot R}\oplus\frac{R}{3R+R_{(-)}\cdot R}\oplus\frac{R}{3R+R_{(-)}\cdot R}.\qed$$
\end{theorem}

\begin{remark}
\label{rmk:p3_hml_com}
If $R$ is super-commutative, then $\mathfrak{p}_3(R,\rho)\cong\mathfrak{p}_3(\Bbbk)\otimes_{\Bbbk}R$. Theorem~\ref{thm:p3_hml} ensures that
$$\mathrm{H}_2(\mathfrak{p}_3(\Bbbk)\otimes_{\Bbbk}R)=(R/3R)\oplus(R/3R)\oplus(R/3R),$$
which vanishes whenever $3$ is invertible in $R$. If $\Bbbk$ is a field of characteristic zero, this coincides with the second homology of $\mathfrak{p}_3(\Bbbk)\otimes_{\Bbbk}R$ given in \cite{IoharaKoga2005}.
\end{remark}
\medskip

In the special case where $(R,{}^-)=(S\oplus S^{\mathrm{op}},\mathrm{ex})$, Theorem~\ref{thm:p3_hml} recovers the second homology of $\mathfrak{sl}_{3|3}(S)$ obtained in \cite{ChenSun2015}.

\begin{corollary}
\label{cor:uce_sl33}
Let $S$ be an arbitrary unital associative superalgebra. Then
$$\mathrm{H}_2(\mathfrak{sl}_{3|3}(S))\cong\mathrm{HC}_1(S).$$
\end{corollary}
\begin{proof}
It is known from Example~\ref{ex:iso_p_SS} that $\mathfrak{sl}_{3|3}(S)$ is isomorphic to $\mathfrak{p}_3(S\oplus S^{\mathrm{op}},\mathrm{ex})$. Hence, 
$$\mathrm{H}_2(\mathfrak{sl}_{3|3}(S))\cong\fourIdx{}{+}{}{1}{\mathrm{HD}}(S\oplus S^{\mathrm{op}},\mathrm{ex}\circ\rho)\oplus\mathfrak{z},$$
in which $\fourIdx{}{+}{}{1}{\mathrm{HD}}(S\oplus S^{\mathrm{op}},\mathrm{ex}\circ\rho)$ is isomorphic to $\mathrm{HC}_1(S)$ as in Corollary~\ref{cor:slmm}. On the other hand, since $(R,{}^-)=(S\oplus S^{\mathrm{op}},\mathrm{ex})$, we observe that $R_{(-)}$ contains a unit element $1\oplus(-1)$. Hence, $\mathfrak{z}=0$ and the desired isomorphism follows. 
\end{proof}

\section*{Acknowledgements}
The authors are grateful to the referees for their valuable comments. The authors thank Prof. Yun Gao and Prof. Hongjia Chen for useful suggestions. Zhihua Chang was supported by the National Natural Science Foundation of China (Grant No. 11501213), the China Postdoctoral Science Foundation (Grant No. 2015M570705). Yongjie Wang was supported by the China Postdoctoral Science Foundation (Grant No. 2015M571928).

\end{document}